\documentclass{smfart}
\usepackage[utf8]{inputenc}
\usepackage{a4wide}
\usepackage{amsmath}
\usepackage{bbm}
\usepackage{amsthm}
\usepackage{amssymb,mathrsfs}
\usepackage{tikz}
\usepackage[english,french]{babel} 
\usepackage[all]{xy}
\usepackage{setspace}
\usepackage{color}
\usepackage{tabularx}
\usepackage{nicefrac}
\usepackage{graphicx}
\usepackage{tikz}
 \usetikzlibrary{arrows,shapes,positioning, calc}

\theoremstyle{plain}
\newtheorem{theorem}{Théorème}[section]
\newtheorem{corollary}[theorem]{Corollaire}

\newtheorem{proposition}[theorem]{Proposition}

\theoremstyle{definition}
\newtheorem{definition}[theorem]{Définition}
\newtheorem*{remark}{Remarque}
\newtheorem{conjecture}[theorem]{Conjecture}

\allowdisplaybreaks
                                                                                                               
\newcommand{\EE}{\mathbb E}

\newcommand\N{{\mathbb N}}
\newcommand\Z{{\mathbb Z}}

\newcommand\R{{\mathbb R}}

\newcommand\DD{{\mathbb D}}

\newcommand{\B}{{\mathcal B}}

\newcommand{\UU}{\sum \limits_{n=1}^N}
\newcommand {\eps}{\varepsilon}
\newcommand {\FAIBLE} {\text{faible}}
\newcommand {\DIM}{\text{dim}}

\newcommand{\demi}{{\nicefrac{1}{2}}}
\newcommand{\err}{\mathrm{R}}

\newcommand {\RADcv}{\text{\rm Rad}}
\newcommand {\GAUSScv} {\gamma}
\newcommand {\RADbdd}{\text{\rm Rad}{}_\infty}
\newcommand {\GAUSSbdd} {\gamma_\infty}
\newcommand {\SUCHTHAT} {\,:\;}

\newcommand {\defequal} {\; \overset{\text{def}}{=\joinrel=} \;}
\newcommand {\norm}[1] {\| #1 \|}               
\newcommand {\bignorm}[1]{\bigl\| #1 \bigr\|}
\newcommand {\Bignorm}[1]{\Bigl\| #1 \Bigr\|}

\newcommand {\ELL} {{\mathcal L}}
\newcommand {\BOUNDED} {{\mathcal L}}
\newcommand {\LIN} {\text{\rm Lin}}
\newcommand {\HS} {\mathrm{S}_2}
\newcommand {\PHS} {\mathrm{PS}_2}
\newcommand {\ESPACEHS} {\mathrm{HS}}

\newcounter{aufzi}
\newenvironment{aufzi}{\begin{list}{ {\upshape\alph{aufzi})}}{
        \usecounter{aufzi}
        \topsep1ex
        \parsep0cm
        \itemsep1ex
        \leftmargin0.8cm
        \labelwidth0.5cm
        \labelsep0.3cm
}}
{\end{list}}

\newcounter{aufzii}
\newenvironment{aufzii}{\begin{list}{\hfill {\upshape(\roman{aufzii})}}{
        \usecounter{aufzii}
        \topsep1ex
        \parsep0cm
        \itemsep1ex
        \leftmargin0.8cm
        \labelwidth0.5cm
        \labelsep0.3cm
         \itemindent0cm
}}
{\end{list}}

\begin{document}

\title[Extensions Banachiques d'opérateurs de Hilbert-Schmidt]{Sur quelques extensions au cadre Banachique de la notion d'opérateur de Hilbert-Schmidt}
\date{\today}
 
\author[S.A. Abdillah]{Said Amana Abdillah} 
\address{Université des Comores\\Rue de la Corniche\\B.P2585 Moroni (Comores)}
\email{intissoir2002@hotmail.fr}

\author[J. Esterle]{Jean Esterle}
\address{Institut de Mathématiques de Bordeaux\\Université Bordeaux 1\\351, cours de la Libération\\33405 Talence CEDEX\\FRANCE}
\email{esterle@math.u-bordeaux.fr}
\author[B.H. Haak]{Bernhard H. Haak}
\address{Institut de Mathématiques de Bordeaux\\Université Bordeaux 1\\351, cours de la Libération\\33405 Talence CEDEX\\FRANCE}
\email{bernhard.haak@math.u-bordeaux.fr}

\begin{abstract}
  Le but de cet article est de faire le point sur diverses
  possibilités connues pour étendre au cadre Banachique la notion
  d'opérateur de Hilbert-Schmidt: opérateurs $p$-sommants,
  $\gamma$-sommants ou $\gamma$-radonifiants, opérateurs faiblement
  $*1$-nucléaires et classes d'opérateurs définies par des
  propriétés de factorisation. On introduit la classe $\PHS(E; F)$
  des opérateurs pré-Hilbert-Schmidt comme étant la classe des
  opérateurs $u:E\to F$ tels que $w\circ u \circ v$ soit
  Hilbert-Schmidt pour tout opérateur borné $v: H_1\to E$ et tout
  opérateur borné $w:F\to H_2$, $H_1$ et $H_2$ désignant des
  espaces de Hilbert quelconque. Hormis le cas trivial où l'un des
  deux espaces $E$ ou $F$ est un "espace de Hilbert-Schmidt", cet
  espace ne semble avoir été décrit que dans le cas banal où
  l'un des deux espaces $E$ et $F$ est un espace de Hilbert.

\bigskip

\noindent {\sc Mots clés}: Espaces de Banach, opérateurs de Hilbert-Schmidt,
opérateurs $p$-sommants, opérateurs presque sommants, opérateurs
$\gamma$-sommants, opérateurs $\gamma$-radonifiants, inégalité
de Grothendieck.

\bigskip

\noindent {\sc Abstract}:  In this work we discuss several ways to extend to the
context of Banach spaces the notion of Hilbert-Schmidt operators:
$p$-summing operators, $\gamma$-summing or $\gamma$-radonifying
operators, weakly $*1$-nuclear operators and classes of operators
defined via factorization properties. We introduce the class
$\PHS(E; F)$ of pre-Hilbert-Schmidt operators as the class of all
operators $u:E\to F$ such that $w\circ u \circ v$ is Hilbert-Schmidt
for every bounded operator $v: H_1\to E$ and every bounded operator
$w:F\to H_2$, where $H_1$ et $H_2$ are Hilbert spaces.  Besides the
trivial case where one of the spaces $E$ or $F$ is a "Hilbert-Schmidt
space", this space seems to have been described only in the easy
situation where one of the spaces $E$ or $F$ is a Hilbert space.

\bigskip

\noindent {\sc Keywords}: Banach spaces, Hilbert-Schmidt operators, $p$-summing
operators, almost summing operators, $\gamma$-summing operators,
$\gamma$-radonifying operators, Grothendieck's inequality.
\end{abstract}

\subjclass{47B10, 46C05, 46B25}

\maketitle

\section{Introduction} 


Le but de cet article, qui fait suite au travail entrepris par le
premier auteur dans sa thèse \cite{Abd}, est de faire le point sur
diverses possibilités d'étendre au cadre Banachique la notion
d'opérateur de Hilbert-Schmidt.

Soient $H_1$ et $H_2$ deux espaces de Hilbert. Les opérateurs de
Hilbert-Schmidt $T:H_1 \to H_2$ sont initialement définis par la
condition
\[
   \sum \limits_{n=1} ^{+\infty}\Vert Th_n\Vert^2<+\infty
\]
pour toute suite orthonormale $(h_n)_{n\ge 1}$ d'éléments de $H_1$,
ce qui est équivalent à la sommabilité de la famille $\Vert
Te_i\Vert^2_{i\in I}$ pour au moins une (donc pour toute) base
hilbertienne $(e_i)_{i \in I}$ de $H_1$.

Ces opérateurs admettent de nombreuses caractérisations bien
connues, par des conditions qui s'expriment naturellement dans le
cadre banachique: caractérisations par des propriétés de
factorisation, ou par l'action de l'opérateur induit sur des suites
vectorielles. On obtient ainsi d'une part les notions d'opérateurs
absolument sommants, $p$-sommants, presque-sommants,
$\gamma$-sommants, $\gamma$-radonifiants, et d'autre part les
opérateurs ayant certaines propriétés de factorisation.  Le tableau
suivant fait apparaître les propriétés d'inclusion
entre ces diverses classes d'opérateurs, qui coïncident
toutes dans le cas hilbertien. La classe la plus générale est
celle des opérateurs pré-Hilbert-Schmidt, c'est à dire la classe
des opérateurs linéaires $T:E \to F$ tels que $S\circ T\circ R:
H_1 \to H_2$ soit Hilbert-Schmidt pour tout opérateur linéaire
borné $R: H_1\to E$ et tout opérateur linéaire borné $S:F\to
H_2,$ $H_1$ et $H_2$ désignant des espaces de Hilbert.  En ce qui
concerne les propriétés de factorisation, rappelons qu'un
opérateur $T:E \to F$ est dit universellement factorisable si $T$
admet une factorisation à travers $G$ pour tout espace de Banach
$G$, et qu'un espace de Banach $L$ est appelé espace de
Hilbert-Schmidt si tout opérateur $T: H_1 \to H_2$ admettant une
factorisation à travers $L$ est un opérateur de Hilbert-Schmidt,
voir \cite{DJT}. Parmi les classes définies par des propriétés
de factorisation qui coïncident avec la classe des
opérateurs de Hilbert-Schmidt dans le cas hilbertien, la plus
restreinte est alors la classe des opérateurs universellement
factorisables et la plus générale est celle des opérateurs
admettant une factorisation de Hilbert-Schmidt. Le fait que les
espaces de type ${\mathcal L}_1$ ou ${\mathcal L}_{\infty}$ sont des
espaces de Hilbert-Schmidt est une conséquence de la célèbre
inégalité de Grothendieck, voir les sections 2 et 3.

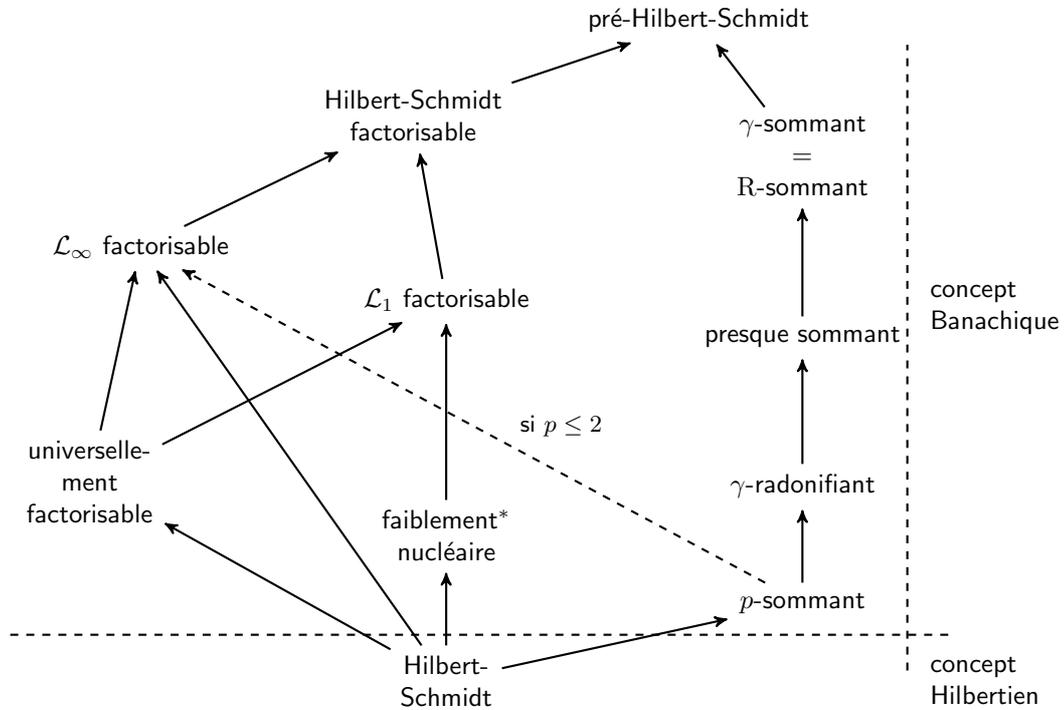
\begin{center}
\begin{tikzpicture}[->,>=stealth',shorten >=1pt,auto, thick,main node/.style={circle,fill=blue!20,draw,font=\sffamily}]
  \node[main node, align=center, rectangle, draw=none,fill=none] (1) {Hilbert-\\Schmidt} ;
  \node[main node, align=center, rectangle, draw=none,fill=none] (2) [above left = 1.5cm and 3cm of 1] {universelle-\\ment\\factorisable} ;
  \node[main node, align=center, rectangle, draw=none,fill=none] (3) [above = 1cm of 1] {faiblement${}^*$\\nucléaire} ;
  \node[main node, align=center, rectangle, draw=none,fill=none] (4) [above right = 0.3cm and 3cm of 1] {$p$-sommant} ;
  \node[main node, align=center, rectangle, draw=none,fill=none] (5) [above left = 5cm and 2cm of 1] {${\mathcal L}_\infty$ factorisable} ;
  \node[main node, align=center, rectangle, draw=none,fill=none] (6) [above =4.3cm of 1] {${\mathcal L}_1$ factorisable} ;
  \node[main node, align=center, rectangle, draw=none,fill=none] (7) [above right = 1cm and 1cm of 5] {Hilbert-Schmidt\\factorisable} ;
  \node[main node, align=center, rectangle, draw=none,fill=none] (8) [above right = 8cm and 1cm of 1] {pré-Hilbert-Schmidt} ;
  \node[main node, align=center, rectangle, draw=none,fill=none] (9) [above = 1cm of 4] {$\gamma$-radonifiant} ;
  \node[main node, align=center, rectangle, draw=none,fill=none] (10)[above = 3cm of 4] {presque sommant} ;
  \node[main node, align=center, rectangle, draw=none,fill=none] (11)[above = 5cm of 4] {$\gamma$-sommant\\ $=$ \\$\err$-sommant} ;
  \node[draw=none, fill=none] (L)[above left  = 0.01cm and 5cm of 1] {} ;
  \node[draw=none, fill=none] (R)[above right =0.01cm and  6cm of 1] {} ;
  \node[draw=none, fill=none] (D)[right = 5.2cm of 1] {} ;
  \node[draw=none, fill=none] (U)[above right = 8cm and 5.2cm of 1] {} ;
  \node[main node, align=left, rectangle, draw=none,fill=none] (Hilbert) [right = 5.5cm of 1] {concept\\Hilbertien} ;
  \node[main node, align=left, rectangle, draw=none,fill=none] (Banach) [above right = 4cm and 5.5cm of 1] {concept\\Banachique} ;
  \path[every node/.style={font=\sffamily\small}]  
   (L) edge [dashed,-] node {} (R)
   (D) edge [dashed,-] node {} (U)
   (1) edge node {}  (2) 
       edge node {}  (5)
       edge node {}  (3)
       edge node {}  (4)
   (2) edge node {}  (5)
   (2) edge node {}  (6)
   (3) edge node {}  (6)
   (4) edge node {}  (9)
       edge [dashed] node [right] {\hspace*{0.5cm}si $p\le 2$} (5) 
   (5) edge node {}  (7)
   (6) edge node {}  (7)
   (7) edge node {}  (8)
   (9) edge node {} (10)
  (10) edge node {} (11)
  (11) edge node {}  (8)
   ;
\end{tikzpicture}
\end{center}

Dans la section~2 on rappelle les définitions de certains espaces de
suites classiques: suites faiblement $\ell p,$ suites Rademacher
sommables, Gauss sommables, et espaces de suites un peu plus
généraux pour les espaces de Banach contenant une copie de $c_0$
caractérisés par la bornitude de sommes partielles, et on
rappelles les inégalités de Khintchine, Kahane et Grothendieck.

Dans la section~3 on détaille les diverses caractérisations des
opérateurs de Hilbert-Schmidt évoquées plus haut, qui font
intervenir les diverses classes d'opérateurs introduites dans le
tableau ci-dessus, et on explique les inclusions et comparaisons
indiquées dans ce tableau. Ces caractérisations sont classiques,
à part sans doute l'identification entre opérateurs de
Hilbert-Schmidt et opérateurs faible$*$ 1-nucléaires, pour
laquelle nous n'avons pas trouvé de référence dans la
littérature.

Dans la section~4 nous rappelons les notions de type et cotype des
espaces de Banach, et le fait que la classe $\gamma(E; F)$ des
opérateurs $\gamma$-radonifiants de $E$ dans $F$ se réduit à la classe
$\Pi_2(E; F)$ des opérateurs $2$-sommants de $E$ dans $F$ si $F$ est
de cotype 2, et même à la classe $\Pi_1(E; F)$ des opérateurs
absolument sommants de $E$ dans $F$ si $E$ est en outre de type
fini. Par contre Linde et Pietsch ont observé dans \cite{lp} que si
$E=F=\ell_{\infty}$ alors $\Pi_q(E; F)$ contient strictement $\Pi_p(E;
F)$ pour $1\le p <q <+\infty$ et les inclusions $\cup_{p\ge 1}(E;
F)\subset \gamma(E; F)\subset \gamma^{\infty}(E; F)$ sont également
strictes. Nous évoquons aussi à la section 4 une confusion
gênante du chapitre~12 de \cite{DJT} entre la classe $\Pi_{ps}(E;F)$
des opérateurs presque-sommants et la classe $R^{\infty}(E;F)$ des
opérateurs Rademacher-sommants, alors que les théorèmes de
Hoffmann-J\o{}rgensen et Kwapie\'n \cite{HK}, \cite{Kwap} qui
garantissent l'égalité entre ces deux classes ne s'appliquent que
si $F$ ne contient aucun espace isomorphe à $c_0.$

A la question~5, après avoir rappelé que tout opérateur de
Hilbert-Schmidt est universellement factorisable et décrit diverses
caractérisations des espaces de Hilbert-Schmidt données dans
\cite{DJT}, on montre que la classe $PS_2(E;F)$ des opérateurs de
Hilbert-Schmidt de $E$ dans $F$ coïncide avec la classe
$\gamma(E;F)$ des opérateurs de $\gamma$-radonifiants de $E$ dans
$F$ si $F$ est de type 2, résultat qui semble nouveau. Des
résultats de Linde et Pietsch, complétés par des calculs de
Maurey mentionnés dans \cite{lp}, donnent une description complète
des classes $\gamma(\ell _p,\ell_q).$ On peut alors en déduire une
description complète des opérateurs pré-Hilbert-Schmidt de
$\ell_p$ dans $\ell_q$ pour $1\le p < 2$, $1 \le q <+\infty$ et pour
$2 \le p < +\infty,$ $2\le q <+\infty$ (il est clair d'autre part que
tout opérateur de $\ell_p$ dans $\ell_q$ est pré-Hilbert-Schmidt
si $\sup(p,q)=+\infty$).

Nous concluons l'article en conjecturant que tout opérateur
pré-Hilbert-Schmidt se factorise à travers un espace de
Hilbert-Schmidt. Il résulte du théorème de factorisation de
Pietsch que cette conjecture est vérifiée dans le cas particulier
des opérateurs $p$-sommants pour $1 \le p \le 2.$

\section{Quelques espaces de suites classiques}

 Les espaces de Banach considérés dans cet article
 sont des espaces de Banach réels. On notera $E $ et $F$ deux espaces
 de Banach et $H$ un espace de Hilbert. Le dual de $E$ sera noté $E^*$
 et la boule unité de $E$ sera notée $B_E$.  Si $u:E \to F$ est un
 opérateur linéaire continu, on définit son adjoint $u^*:F^* \to E^*$
 par la formule $\left\langle x,u^*(y^*)\right\rangle=\left\langle
   y^*,u(x)\right\rangle$, de sorte que
 $\left\|u^*\right\|=\left\|u\right\|$. On note $\BOUNDED(E; F)$
 l'espace des opérateurs bornés de $E$ dans $F$, et pour $x^*\in E^*$,
 $y \in F$ on définit $x^*{\otimes} y \in \BOUNDED(E; F)$ par la
 formule
\[
    (x^*{{\otimes}} y) \, x=\langle x,x^* \rangle y \ \ \ \ (x\in E).
\]
Pour $p > 1$ on note $p^*$ le conjugué de $p$, défini par la formule
$\frac{1}{p}+\frac{1}{p^*}=1$, avec la convention $p^*=+\infty$ si
$p=1$. Pour $p>1$ on note $(\ell_p(E), \Vert \cdot \Vert_p)$ l'espace de Banach
des suites absolument $p$-sommables d'éléments de $E$, on note 
$(\ell_{\infty}(E),\Vert \cdot \Vert)_{\infty}$  l'espace  des suites
bornées d'éléments de $E$, et on note $c_0(E)$ le sous-espace fermé de
$\ell_{\infty}(E)$ formé des suites qui convergent vers $0$. Lorsque
$E=\mathbb{R}$ ces espaces sont simplement notés $\ell_p, 
\ell_{\infty}$ et $c_0$; pour des familles $(x_i)_{i\in I}$, indexés
par $i \in I$ on indique si besoin est l'ensemble des indices: par exemple
$\ell_{p}(\Z,\R)$ désigne  les suites absolument $p$-sommables de réels indexées sur les entiers relatifs,
$\ell_{p}(\Z,  E)$ les suites absolument $p$-sommables d'éléments de $E$ indexées sur les entiers
relatifs, etc.

\begin{definition}\label{def:weak-lp}
  Soit $p \in [1, \infty]$. Une suite $(x_n)_{n\ge 1}$ d'un espace de
  Banach $E$ est appelée faiblement $p$-sommable si
\begin{equation}\label{eq:lp-faible}
  \left\|(x_n)_{n\ge 1}\right\|_{\ell_p^\FAIBLE(E)}    \defequal    \sup_{x^*\in B_{E^*}}\left(\sum_{n=1}^\infty\left|\left\langle x_n,x^*\right\rangle\right|^p\right)^{\nicefrac{1}{p}}
\end{equation}
est fini.
\end{definition}

\noindent L'espace vectoriel des suites d'éléments de $E$ qui sont
faiblement $p$-sommables sera noté $\ell_p^\FAIBLE(E)$. Muni de la
norme~(\ref{eq:lp-faible}), $\ell_p^\FAIBLE(E)$ est un espace de
Banach, voir par exemple \cite[page 32]{DJT}. Rappelons qu'une partie
$F\subset E^*$ est {\it normante} quand $\left \Vert x \right \Vert
=\sup_{x^*\in F}\left | \langle x,x^* \rangle \right |$ pour tout $x
\in E$. Si $F$ est un sous-ensemble normant de $E^*$, on vérifie
facilement que l'on a
\begin{equation} \label{lp-faible-via-normant}
  \Vert (x_n)_{n\ge 1}\Vert_{\ell_p^\FAIBLE(E)}=\sup_{x^*\in F}  \Bigl(\sum_{n=1}^\infty \left | \langle  x_n,x^* \rangle \Bigr | ^p\right )^{\nicefrac{1}{p}}.
\end{equation}


\noindent 
\begin{definition}
  Une série $\sum_n x_n$ dans $E$ est appelée {\it commutativement
    convergente} si la série $\sum_n x_{\sigma(n)}$ converge pour
  toute permutation $\sigma:\N \to \N$.
\end{definition}

\noindent La convergence commutative équivaut à la convergence de la
série $\sum_n \epsilon_n x_n$ pour toute suite $(\epsilon_n)_{n\ge
  1}\in \{-1,1\}^{\N}$, voir par exemple \cite[Théorème 1.9]{DJT}. Dans
ce cas
\begin{equation}  \label{eq:omnibus-thm-uncond-summable}
   \sup_{\eps_n = \pm 1}  \Bignorm{ \sum_n \epsilon_n x_n } 
\end{equation}
est fini (ibidem). Observons que pour tout $N \ge 1$, on a
\begin{align*}
 \sup_{\eps_n = \pm 1}  \Bignorm{ \sum_{n=1}^N \epsilon_n x_n } 
= & \;  \sup_{\eps_n = \pm 1}  \sup_{\norm{x^*}\le 1 } \sum_{n=1}^N  \epsilon_n \langle x_n, x^* \rangle  \\
= & \;  \sup_{\norm{x^*}\le 1 } \sup_{\eps_n = \pm 1} \sum_{n=1}^N  \epsilon_n \langle x_n, x^* \rangle  
= \sup_{\norm{x^*}\le 1 } \sum_{n=1}^N  \bigl| \langle x_n, x^* \rangle \bigr|.
\end{align*}
Ainsi, la quantité (\ref{eq:omnibus-thm-uncond-summable}) n'est rien que la norme
$\ell_1^\FAIBLE(E)$ de la suite $(x_n)$.

\bigskip

\noindent On va maintenant introduire les suites Gauss-sommables et
Rademacher-sommables. Dans la suite de
l'article, on notera $(\gamma_n)_{n\ge 1}$ une suite de variables 
gaussiennes indépendantes, et $(r_n)_{n\ge 1}$ une suite de variables
Rademacher indépendantes, c'est à dire une suite de variables
aléatoires indépendantes prenant les valeurs $\pm 1$ avec la
probabilité $\demi$. Une construction explicite d'une telle suite est
donné par $r_n(t) = \text{sign} \bigl( \sin(2^n \pi t)\bigr)$ sur
$\Omega=[0,1]$.

\begin{definition}\label{def:RAD}
  Une suite $(x_n)_{n\ge 1}$ d'un espace de Banach $E$ est appelée
  Rademacher-sommable, et on note  $(x_n) \in \RADcv(E)$, si la série $\sum
  r_n x_n$ converge dans $\ELL_2(\Omega; E)$. Dans ce cas on pose
\[
   \norm{  (x_n)_n }_{\RADcv(E)} \defequal   \Bigl( \EE \Bignorm{ \sum_{n} r_n x_n }^2 \Bigr)^\demi.
\]
\end{definition}

\noindent Rappelons que si $X$ et $Y$ sont des variables aléatoires
indépendants et symétriques, alors $\EE \norm{ X }^2 \le \EE \norm{ X +
  Y}^2$. Ainsi, les (deuxièmes) moments des sommes partielles de la
série $\sum r_n x_n$ sont croissants. Cette observation amène à
étudier l'espace $\RADbdd(E)$ des suites $(x_n)$ d'éléments de $E$ dont les sommes partielles sont
uniformément bornées en norme $\ELL_2$:
\[
   \norm{  (x_n)_n }_{\RADbdd(E)} \defequal  \sup_{N \ge 1} \Bigl( \EE \Bignorm{ \sum_{n = 1}^N r_n x_n }^2 \Bigr)^\demi
\]
La base canonique de $c_0$ montre que la bornitude uniforme des sommes
partielles n'implique pas la convergence; en effet on voit dans cet
exemple que $\EE \bignorm{ \sum_{n=N}^M r_n e_n }^2 = 1$, les sommes
partielles ne forment donc pas une suite de Cauchy. 

\begin{definition}\label{def:GAUSS}
  Une suite $(x_n)_{n\ge 1}$ d'un espace de Banach $E$ est appelée
  Gauss-sommable, et on note $(x_n) \in \GAUSScv(E)$, si la série $\sum
  \gamma_n x_n$ converge dans $\ELL_2(\Omega; E)$. Dans ce cas on pose
\[
   \norm{  (x_n)_n }_{\GAUSScv(E)} \defequal   \Bigl( \EE \Bignorm{ \sum_{n} \gamma_n x_n }^2 \Bigr)^\demi.
\]
\end{definition}

\noindent De manière analogue on considère aussi l'espace $\GAUSSbdd(E)$ des suites $(x_n)_{n\ge 1}$ d'éléments de $E$ qui satisfont
\[
   \norm{  (x_n)_n }_{\GAUSSbdd(E)} \defequal  \sup_{N\ge 1} \Bigl( \EE \Bignorm{ \sum_{n=1}^N \gamma_n x_n }^2 \Bigr)^\demi < +\infty,
\]
ce qui n'entraîne pas nécessairement la convergence de la série $\sum_{n=1}^{+\infty}\gamma_nx_n.$  On reviendra sur la distinction
entre convergence et bornitude uniforme des sommes partielles plus
loin; notons simplement que l'espace $c_0$ joue un rôle crucial, comme le
montre le théorème de Hoffmann-J\o{}rgensen et Kwapień, \cite{HK}.

\noindent Les espaces de suites $\GAUSSbdd(E)$ et $\RADbdd(E)$ sont des espaces de Banach. 
 En effet, toute suite de Cauchy
 $(x^{(n)})$ de $\GAUSSbdd(E)$ est bornée. Appelons donc $C>0$ le supremum
 des normes et notons $x^{(n)} = (x^{(n)}_k)_{k\ge 1}$. La suite
 $(x^{(n)}_k)_{n\ge 1}$ est de Cauchy et converge donc pour
 $n\to+\infty$ vers un $x_k$. On conclut avec le lemme de Fatou que pour
 tout ensemble fini $K\subset \N$
 \[
   \EE \Bignorm{ \sum_{k \in K} \gamma_k x_k  }^2 \le \liminf_n \EE \Bignorm{ \sum_{k \in K} \gamma_k x^{(n)}_k  }^2 \le C.
 \]
 Ainsi, $x = (x_k) \in \GAUSSbdd(E)$. Une autre application du lemme de
 Fatou donne facilement la convergence dans la norme de $\GAUSSbdd(E)$.
 En remplaçant $\gamma_n$ par $r_n$ la même preuve donne la complétude
 de $\RADbdd(E)$.  
\noindent  Cette propriété se transmet à leurs sous-espaces fermés $\GAUSScv(E)$ et
$\RADcv(E)$.

\bigskip

Nous terminons cette partie en rappelant trois inégalités classiques;
pour leurs démonstrations on renvoie par exemple à \cite[Théorème 1.10, 1.14 et
11.1]{DJT}.

\begin{theorem}[Inégalités de Khintchine]\label{thm:hinchin}
  Pour tout $0 < p < +\infty$, il existe des constantes positives
  $A_p$ et $B_p$ telles que pour toute suite $\bigl(a_n\bigr)_{n\ge
    1}\in\ell_2$, on ait :
\[
A_p\Bigl(\sum \limits_ {n=1}^{+\infty} \bigl|a_n\bigr|^2\Bigr)^\demi
\leq
\Bigl(\int_0^1 \Bigl|\sum _n  a_n r_n(t) \Bigr|^p dt \Bigr)^{\nicefrac{1}{p}}
\leq
B_p\Bigl(\sum \limits_ {n=1}^{+\infty} \bigl|a_n\bigr|^2\Bigr)^\demi.
\]  
\end{theorem}
 
\noindent En général, on ne peut pas remplacer la valeur absolue par la norme
d'un espace de Banach. Cependant, on a l'inégalité de Kahane (voir par
exemple \cite[Théorème 11.1]{DJT}) :

\begin{theorem}[Inégalités de Kahane]\label{thm:kahane}
  Soit $0 < p,q < \infty$. Il existe alors une constante $K_{p,q}$
  telle que pour tout espace de Banach $E$ et $x_1, \ldots, x_N \in E$ on ait
\[
   \Bigl( \EE \Bignorm{ \sum_{n=1}^N r_n x_n }^q \Bigr)^{\nicefrac1q} \le K_{p,q} \Bigl( \EE \Bignorm{ \sum_{n=1}^N r_n x_n }^p \Bigr)^{\nicefrac1p}
\]
et 
\[
   \Bigl( \EE \Bignorm{ \sum_{n=1}^N \gamma_n x_n }^q \Bigr)^{\nicefrac1q} \le K_{p,q} \Bigl( \EE \Bignorm{ \sum_{n=1}^N \gamma_n x_n }^p \Bigr)^{\nicefrac1p}.
\]
\end{theorem}

\noindent La version Gaussienne des inégalités de Kahane se démontre
élégamment à partir de la version Rademacher en utilisant le théorème central  limite.

\begin{theorem}[Inégalité de Grothendieck]\label{thm:grothendieck-ineq}
  Soit $H$ un espace de Hilbert réel de dimension $n$, 
  $(a_{ij})_{i,j\leq 1}$ une matrice $n{\times}n$ et soient $x_1,
  \ldots,x_n, y_1, \ldots ,y_n \in B_H$. On a:
\[
    \Bigl|\sum_{i,j} a_{ij}\bigl\langle x_i,y_j\bigr\rangle\Bigr|
\leq 
    K_G \sup \Bigl\{\; \Bigl|  \sum_{i,j}a_{ij} s_i t_j \Bigr| \SUCHTHAT \; |s_i|\leq 1, |t_j| \leq 1 \Bigr\}
\] 
où $K_G$ est une constante universelle appelée constante de
Grothendieck.
\end{theorem}

\noindent Il est important de noter que, $\ell_{\infty}^n$ et
$\ell_1^n$ étant considérés au sens réel, on a
\begin{align*}
\left\|\left(a_{ij}\right)\right\|_{\ell_{\infty}^n\rightarrow\ell_{1}^n}\
=& \; \sup_{\left|y_j\right|\leq 1}\Bigl\|\sum_j a_{ij}y_j\Bigr\|_{\ell_1^n}
= \; \sup_{\left|x_i\right|\leq 1} \sup_{\left|y_j\right|\leq 1}\sum_{i,j}a_{ij}x_i y_j\\
=& \; \sup_{x_i=\pm1}\sup_{y_j=\pm1}\sum_{i,j}a_{ij}x_i y_j.
\end{align*}

\section{Opérateurs de Hilbert-Schmidt et leurs généralisations}

\begin{definition}
  Soit $u \in {\mathcal L}(H_1; H_2)$ un opérateur linéaire. On dit que $u$ est un
  opérateur Hilbert-Schmidt s'il existe une base orthonormale
  $(e_i)_{i\in I}$ de $H_1$ telle que $u e_i \in \ell_2(I; H_2)$.
  L'ensemble des opérateurs de Hilbert-Schmidt est noté
  $S_2(H_1; H_2)$.
\end{definition}

\noindent Il est facile de voir que $S_2(H_1,H_2)$ est un idéal de
$\BOUNDED(H_1; H_2)$.  On peut résumer la théorie classique des
opérateurs de Hilbert-Schmidt, développée par exemple dans \cite{DJT},
par le théorème suivant.

\begin{theorem}\label{thm:HS}
  Soient $H_1$ et $H_2$ deux espaces de Hilbert de dimension infinie,
  et $u \in {\mathcal L}(H_1;H_2)$. Alors les
  conditions suivantes sont équivalentes.
\begin{aufzi}
  \item\label{item:HS1} $u$ est Hilbert-Schmidt.
  \item\label{item:HS2} Pour toute base orthonormale $(e_i)_{i\in I}$
    de $H_1$, on a $u e_i \in \ell_2(I; H_2)$ .
  \item\label{item:HS3} Pour toute famille orthonormale $(e_i)_{i \in
      I}$ d'éléments de $H_1$ et pour toute famille orthonormale
    $(f_j)_{j \in J}$ d'éléments de $H_2$ on a $\bigl( \langle
    u(e_i),f_j \rangle\bigr) \in \ell_2(I{\times}J; \R)$.
  \item\label{item:HS4} Il existe une suite orthonormale $(e_n)_{n\ge
      1}$ d'éléments de $H_1$, une suite orthonormale $(f_n)_{n\ge 1}$
    d'éléments de $H_2$, et une suite $(\tau_n)_{n\ge 1}\in \ell_2$
    telles que l'on ait
    \[ 
        ux=\sum _n\tau_n\langle x,e_n \rangle f_n \ (x \in H_1).
    \] 
  \item\label{item:HS5} Pour toute suite $(x_n)_{n\ge 1} \in
    \ell_2^\FAIBLE(H_1)$, on a $(ux_n)_{n\ge 1}  \in \ell_2(H_2)$.
  \end{aufzi}
  Dans ce cas on a
\[
\sum_{i\in I}\Vert u(e_i)\Vert^2= \sum_{n=1}^{+\infty}a_n^2(u)= \sum
_{n=1}^{+\infty}\vert \tau _n \vert^2 = \inf \Bigl\{ \sum_{n=1}^{+\infty}
\bignorm{ u x_n }^2 \SUCHTHAT \; \norm{ (x_n)_{n\ge 1} }_{\ell_2^\FAIBLE(H_1)}
\le 1 \Bigr\}
\]
où $a_n(u)= \inf \left \{ \Vert u - v\Vert : v \in {\mathcal
    L}(H_1,H_2), rang(v) < n \right \}$ pour $n \ge 1$, et où
$(\tau_n)_{n\ge 1}$ est la suite introduite en \ref{item:HS4}.  
\end{theorem}

\subsection{Opérateurs $p$-sommants}

\begin{definition}\label{def:p-summing}
  Soient $E$ et $F$ deux espaces de Banach, et soit $p \in
  [1,+\infty[$. On dit qu'un opérateur $u \in \LIN(E; F)$ est
  $p$-sommant si $u(x_n) \in \ell_p(F)$ pour toute suite $(x_n) \in
  \ell_p^\FAIBLE(E)$.
\end{definition}

\noindent En utilisant le théorème du graphe fermé, on vérifie qu'un opérateur
$u:E\to F$ est $p$-sommant si et seulement si il existe une constante
$c\geq 0$ telle que pour $m\in\mathbb{N}$, $x_1,\dots,x_N \in E$, on
ait:
\begin{equation}\label{eq:p-summing}
  \sum_{n=1}^{N}\left\|ux_n\right\|^{p}\leq c^{p}\sup_{\left\|x^*\right\|\leq1}\sum_{n=1}^{N}\left|\left\langle
      x_n,x^*\right\rangle\right|^{p}.
\end{equation} 
et dans ce cas on note $\pi_p(u)$ le infimum des constantes $c>0$
vérifiant (\ref{eq:p-summing}). Autrement dit, l'application $\tilde
u: (x_n)_{n\ge 1} \to (ux_n)_{n\ge 1}$ est une application continue de
$\ell_p^\FAIBLE(E)$ dans $\ell_p(E)$, et $\pi_p(u)= \Vert \tilde u
\Vert$, voir \cite[Proposition 2.1.]{DJT}.  On note
$\Pi_p\left(E;F\right)$ l'espace des opérateurs $p$-sommants $u:E \to
F$. Muni de la norme $\pi_p$, $\Pi_p\left(E;F\right)$ est un espace de
Banach \cite[Proposition 2.6, p. 38]{DJT}.  

On déduit de l'inégalité
de Hölder que pour $1 \le p < q < +\infty$, on a  $\Pi_p(E; F)\subset
\Pi_q(E; F)$, et $\pi_q(u)\le \pi_p(u)$ pour $u \in \Pi_p(E; F)$, voir \cite[Théorème
2.8]{DJT}.  De
plus il est immédiat que $\Pi_p(E; F) \subseteq \BOUNDED(E; F)$ avec
$\norm{u}_{\BOUNDED(E; F)} \le \pi_p(u)$.  On a alors la propriété
d'idéal \cite[p. 37]{DJT}, qui résulte immédiatement de la définition
ci-dessus.

\begin{proposition}[propriété d'idéal]
  Soient $E, F, Z, W$ des espaces de Banach, soit $u \in \BOUNDED(E; F)$
  un opérateur $p$-sommant, et soient $v \in \BOUNDED(Z; E)$ et $w \in
  \BOUNDED(F ; W)$ deux opérateurs linéaires continus. Alors
  $w{\circ}u{\circ}v$ est $p$-sommant et \ $\pi_p\left(w{\circ}
    u{\circ}v\right) \leq
  \left\|w\right\|\pi_p\left(u\right)\left\|v\right\|$.
\end{proposition}

Si $K$ est un sous-ensemble faible$*$-fermé normant de $B_{E^*}$,
alors $K$ est faible$*$-compact et l'application $\iota_E: E \to {\mathcal
  C}(K)$ est une isométrie, où ${\mathcal C}(K)$ désigne
l'algèbre de Banach des fonctions continues sur $K$ et où
$\iota_E(x)(x^*)=\langle x,x^* \rangle $ pour $x \in E$, $x^*\in K$.
Un exemple d'opérateur $p$-sommant est l'injection $j_p:
{\mathcal{C}}(K)\to {\mathcal L}^p(\mu):={\mathcal L}^p(K,\mu)$ où ${\mathcal C}(K)$ désigne
l'espace des fonctions continues sur un compact $K$ et $\mu$ une
mesure de probabilité sur $K$, c'est à dire une mesure de Radon
positive sur $K$ telle que $\mu(K)=1$. Le caractère fondamental de cet
exemple est montré par le théorème suivant, voir
\cite[Théorème~2.13]{DJT}, qui joue un rôle essentiel dans la théorie
des opérateurs $p$-sommants.

 \begin{theorem}[Théorème de factorisation de Pietsch]\label{thm:pietsch}
   On suppose que $1\leq p <+\infty$. Soit $u \in \LIN(E ; F)$ un
   opérateur linéaire et $K$ un sous-ensemble normant faible
   $\ast$-fermé de $B_{E^*}$. Les assertions suivantes sont
   équivalentes.
\begin{aufzi}
\item $u$ est $p$-sommant.
\item  Il existe une mesure de probabilité $\mu$ sur K, un sous-espace
  vectoriel fermé $E_p \ de \ \ELL_p\left(\mu\right)$ et un opérateur
  $\hat{u} \in \BOUNDED(E_{p}; F)$ tels que:
 \begin{aufzii}
  \item $ j_{p}\circ \iota_{E}\left(E\right)\subset E_p$ et
  \item $\hat{u}\circ j_{p}\circ \iota_{E}\left(x\right)=ux$
  $\forall$ $x\in E$.  Autrement dit le diagramme suivant commute.
 \[
 \shorthandoff{;:!?}
\xymatrix{ 
    E   \ar[r]^{\iota_E} \ar[d]_{u}   &   \; \iota_E(E) \; \ar[d]_{j_p^E}  \ar@{^{(}->}[r]    & C(K) \ar[d]^{j_p}\\
    F                           &   \; E_p \; \ar[l]^{\widetilde u} \ar@{^{(}->}[r]   & \ELL_p(\mu) \\
}
 \]
\end{aufzii}
\end{aufzi}
\end{theorem}

\begin{corollary} \label{cor:pietsch}
  Soient $E$ et $F$ deux espaces de Banach, et soit $K$ un
  sous-ensemble faible$ *$ fermé normant de $B_{E^*}$. Un
  opérateur linéaire $u \in \BOUNDED(E; F)$ est $2$-sommant si et
  seulement si il existe une mesure de probabilité $\mu$ sur $K$ et
  $\tilde{u}\in \mathcal{L}\left(\ELL_p\left(\mu\right),F\right)$ tel que
  le diagramme suivant commute
 \[
 \shorthandoff{;:!?}\xymatrix{
 E \ar[r]^u \ar[d]_{\iota_E} \ar[r] &F\\
 C\left(K\right) \ar[r]_{j_2} &\ELL_2\left(\mu\right) \ar[u]_{\tilde{u}}
 }
 \]
 et on a dans ce cas $\left\|\tilde{u}\right\|=\pi_2\left(u\right)$.
\end{corollary}

\begin{proof} Ceci se déduit du théorème de factorisation de
  Pietsch dans le cas $p=2$, en posant $\tilde u = \hat u \circ P$,
  où $P$ désigne la projection orthogonale de $L^2(\mu)$ sur
  $E_2$.
\end{proof}

\subsection{Opérateurs absolument sommants}

\begin{definition}\label{def:abs-summing}
  Un opérateur $u:E\to F$ est appelé {\it absolument sommant} si la
  série $\sum_n \norm{ u x_n }$ converge pour toute série
  commutativement convergente $\sum x_n$  d'éléments de $E$.
\end{definition}

\noindent L'ensemble des opérateurs absolument sommants est noté
$\Pi_{abs}(E; F)$, et les éléments de $\Pi_{abs}(E; F)$ sont
caractérisés par la condition
\[
    \pi_{abs}(u) \defequal \sup \; \sum\limits_{n=1}^N  \left \Vert ux_n \right \Vert <+\infty,
\]
le supremum étant calculé sur toutes les familles finies
$(x_1,\dots,x_N)$ d'éléments de $E$ telles que
\[
\sup_{\epsilon_1=\pm 1, \dots, \epsilon_N=\pm 1}\Bignorm{ \sum  \limits _{n=1}^N \epsilon_nux_n } \le 1.
\]
Comme l'ensemble $\{-1,1\}^{\N}$ est un sous-ensemble normant de la
boule unité de $\ell_{\infty}=\ell_1^*$, on obtient facilement le
résultat suivant.

\begin{proposition} \label{pro:abs-sommant-egal-1-sommant}
  Soit $E, F$ deux espaces de Banach, et soit $u \in {\mathcal
    L}(E; F)$. Alors $u$ est absolument sommant si et seulement si $u$
  est $1$-sommant, et dans ce cas $\pi_1(u)=\pi_{abs}(u)$.
\end{proposition}

\noindent En particulier, $\Pi_{abs}(E; F)$ est un idéal d'opérateurs.
On a vu dans l'équivalence \ref{item:HS5} du Théorème~\ref{thm:HS} que
des opérateurs Hilbert-Schmidt coïncident avec les opérateurs
$2$-sommants.  En utilisant le Corollaire~\ref{cor:pietsch} et
l'inégalité de Khintchine on obtient le résultat suivant, voir
\cite[Théorème~2.21.]{DJT}.

\begin{theorem} Soit $E$ un espace de Banach, soit $H$ un espace de
  Hilbert, et soit $u \in \BOUNDED(E;H)$. S'il existe $p \ge 1$
  tel que $u^*$ soit $p$-sommant, alors $u$ est absolument sommant et
  on a
\[
    \pi_1(u)\le A_1^{-1}\, B_p\,  \pi_p(u^*),
\]
où $A_1$ et $B_q$ sont les constantes intervenant dans les
inégalités de Khintchine.
\end{theorem}

\noindent Il résulte de la condition \ref{item:HS4} du
Théorème~\ref{thm:HS} que si $u:H_1\to H_2$ est un opérateur de
Hilbert-Schmidt, alors $u^*:H_2\to H_1$ l'est aussi. On en déduit le
corollaire suivant.

\begin{corollary} \label{coro:sur-Hilbert-p-sommant-egal-HS}
  Soit $u\in\BOUNDED(H_1;H_2)$. Alors les conditions suivantes
  sont équivalentes
  \begin{aufzi}
  \item \label{item:p-sommant-1} Il existe $p \ge 1$ tel que $u$ soit $p$-sommant.
  \item \label{item:p-sommant-2} $u$ est absolument sommant (donc $p$-sommant pour tout $p\ge 1$).
  \item \label{item:p-sommant-3} $u$ est un opérateur de  Hilbert-Schmidt.
  \end{aufzi}
\end{corollary}

\begin{proof} Si $u=u^{**}$ est $p$-sommant pour un réel $p\ge 1$,
  alors $u^*$ est absolument sommant, donc $u$ l'est aussi. D'autre
  part si $u$ est absolument sommant il est $2$- sommant, donc c'est
  un opérateur de Hilbert-Schmidt. Le fait que
  \ref{item:p-sommant-3} implique \ref{item:p-sommant-1} résulte du
  fait que tout opérateur de Hilbert-Schmidt est $2$-sommant.
\end{proof}

\subsection{Inégalité de Grothendieck  et factorisations des opérateurs de Hilbert-Schmidt}

L'inégalité de Grothendieck (Théorème~\ref{thm:grothendieck-ineq}) permet
d'obtenir les majorations classiques suivantes.

\begin{corollary} Soient $n$ et $N$ deux entiers positifs, et soit $p\in [1,2]$.
\begin{aufzi}
\item Pour tout opérateur $u: \ell_1^n\to \ell_2^N$ on a $\pi_1(u)\le K_G\Vert u \Vert$.
\item Pour tout opérateur $v: \ell_{\infty}^n\to \ell_p^N$ on a $\pi_2(u)\le K_G \Vert v \Vert$.
\end{aufzi}
\end{corollary}

On dit qu'un espace de Banach $E$ est $\mathcal{L}_{p,\lambda}$, avec
$1\leq p\leq\infty$ et $1\leq\lambda <+\infty$, si pour tout
sous-espace vectoriel $U$ de dimension finie de $E$, il existe un
sous-espace vectoriel $V$ de dimension finie de $E$ contenant $U$ et
un isomorphisme $\phi:V\rightarrow \ell_{p}^{\DIM(F)}$ tel que
$\left\|\phi\right\|\left\|\phi^{-1}\right\|\leq\lambda$. On dit que
$E$ est un espace $\mathcal{L}_{p}$ s'il existe $\lambda \ge 1$ tel
que $E$ soit $\mathcal{L}_{p,\lambda}$. Il résulte de \cite[Théorème
3.2]{DJT} que si $(\Omega, {\mathcal B}, \mu)$ est un espace mesuré
alors $\ELL_p(\Omega, \mu)$ est un espace $\ELL_{p,\lambda}$ pour tout
$\lambda>1$ si $1\le p \le {\infty}$. De même si $K$ est compact alors
${\mathcal C}(K)$ est un espace $\ELL_{\infty, \lambda}$ pour tout
$\lambda >1$, ce qui implique le résultat ci-dessus pour $p=\infty$
puisque la transformation de Gelfand est un isomorphisme de
$\ELL_{\infty}(\Omega,\mu)$ sur ${\mathcal C}\left (\widehat
  {\ELL_{\infty}(\Omega, \mu)}\right )$ où $\widehat {\ELL_{\infty}(\Omega,
  \mu)}$ désigne l'espace compact formé des caractères de l'algèbre de
Banach $\ELL_{\infty}(\Omega,\mu)$, voir par exemple \cite{DAL}.

Le célèbre théorème de Grothendieck qui est un corollaire de
l'inégalité de Grothendieck (théorème~\ref{thm:grothendieck-ineq})
démontre que tout opérateur linéaire continu $u \in \BOUNDED(\ell_1 ;
\ell_2)$ est absolument sommant. On a plus généralement le résultat
suivant, dont on trouvera une démonstration dans \cite[Chapitre
3]{DJT}.

\begin{theorem} \label{thm:grothendieck}
  \begin{aufzi}
  \item Soit $E$ un espace $\mathcal{L}_{1,\lambda}$ et soit $F$ un
    espace $\mathcal{L}_{2,\mu}$. Alors tout opérateur linéaire borné
    $u: E \to F$ est absolument sommant et on a
\[
    \pi_1(u)\le  \lambda \, \mu \, K_G\, \Vert u \Vert _{\BOUNDED(E; F)}.
\]
\item Soit $E$ un espace $\mathcal{L}_{\infty,\lambda}$ et soit $F$ un
  espace $\mathcal{L}_{p,\mu}$ avec $0 \le p \le 2$. Alors tout
  opérateur linéaire borné $v: E \to F$ est $2$-sommant et on a
\[
    \pi_2(u)\le \lambda \, \mu \, K_G \, \Vert v \Vert _{\BOUNDED(E; F)}.
\]
\end{aufzi}
\end{theorem}

\noindent On obtient alors une caractérisation classique des opérateurs de
Hilbert-Schmidt qui s'exprime en termes de factorisation, voir
\cite[page 85]{DJT}.

\begin{corollary}\label{cor:HS_equivalences}
Soit $u: \BOUNDED(H_1; H_2)$ un opérateur. Les assertions suivantes sont équivalentes.
\begin{aufzi}
 \item\label{item:HS_equivalences-1} $u$ est un opérateur de Hilbert-Schmidt.
 \item\label{item:HS_equivalences-2}  $u$ est factorisable à   travers un espace $\mathcal{L}_\infty$.
 \item\label{item:HS_equivalences-3} $u$ est factorisable à   travers un espace $\mathcal{L}_1$.
 \end{aufzi}
\end{corollary}
\begin{proof}
  \ref{item:HS_equivalences-1} $\Leftrightarrow$
  \ref{item:HS_equivalences-2}: Si $u$ est de Hilbert-Schmidt, alors
  $u$ est $2$-sommant. Donc $u$ est factorisable à travers un espace
  $C\left(K\right)$ d'après le théorème de factorisation de Pietsch,
  et \ref{item:HS_equivalences-2} est vérifié puisque ${\mathcal
    C}(K)$ est un espace $\ELL_{\infty}$. D'autre part si $u$ est
  factorisable à travers un espace $\mathcal{L}_\infty$, alors $u$ est
  $2$-sommant d'après le théorème~\ref{thm:grothendieck} et le
  principe d'idéal. Par conséquent $u$ est de Hilbert-Schmidt.

  \ref{item:HS_equivalences-1}$\Leftrightarrow$~\ref{item:HS_equivalences-3}:
  Cette équivalence entre est une version duale de la première
  équivalence. On sait que le dual de l'espace ${\mathcal C}\left(K\right)$ est
  un espace $\mathcal{L}_1$. Si $u \in {\mathcal L}(H_1; H_2)$ est de
  Hilbert-Schmidt, alors $u^*$ est de Hilbert-Schmidt. D'après la
  condition \ref{item:HS4} du théorème~\ref{thm:HS}, $u^*$ admet donc
  une factorisation de la forme: 
$u^*:\; H_2 \longrightarrow   {\mathcal C}\left(K\right) \longrightarrow H_1$ et $u=u^{**}$ admet une
  factorisation 
$u:\;  H_1\longrightarrow\mathcal{L}_1\longrightarrow H_2$.
  Réciproquement, si $u:H_1\rightarrow H_2$ est factorisable à travers
  un espace $\mathcal{L}_1$, alors on déduit du
  théorème~\ref{thm:grothendieck} que $u$ est $2$-sommant, donc c'est
  un opérateur de Hilbert-Schmidt.
\end{proof}

\subsection{Opérateurs de Hilbert-Schmidt et nucléarité}


 \begin{definition}
   Soit $1\leq p < \infty$.  Un opérateur $u \in \BOUNDED(E; F)$ est
   $p$-nucléaire s'il existe des opérateurs
   $a\in\mathcal{L}(\ell_p;F)$, $b\in\mathcal{L}(E; \ell_\infty)$ et
   une suite $\lambda=(\lambda_n)_{n\ge 1}\in\ell_p$ tels que
   $u=a\circ M_{\lambda}\circ b$ où $M_{\lambda} \in
   \BOUNDED(\ell_\infty; \ell_p)$ désigne l'opérateur de
   multiplication par $\lambda$.

   \noindent On note $\mathcal{N}_p\left(E;F\right)$ l'ensemble des
   opérateurs $p$-nucléaires de $E$ dans $F$ et on pose
   $\nu_p\left(u\right)=\inf\left\|a\right\|\left\|\lambda\right\|_{\ell_p}\left\|b\right\|$,
   l'infimum étant pris sur toutes les factorisations du type
   ci-dessus.
\end{definition}

\noindent Les opérateurs $1$-nucléaires sont souvent appelés 
opérateurs nucléaires.

\begin{definition}
Soit $u\in\mathcal{L}\left(E;F\right)$.
\begin{aufzi}
\item On dit que $u$ est faiblement $p$-nucléaire si $u=\sum_n x_n^*{\otimes} y_n$ avec $(x_n^*)_{n\ge
    1}\in\ell_p^\FAIBLE(E^*)$ et $(y_n)_{n\ge 1}\in\ell_q(F)$. 
\item On dit que $u$ est faiblement${}^\ast$ $p$-nucléaire si $u=\sum_n
  x_n^*{\otimes} y_n$ avec $(x_n^*)_{n\ge
    1}\in\ell_p^{\FAIBLE\ast}(E^*)$ et $(y_n)_{n\ge 1}\in\ell_q(F)$.
\end{aufzi}
\end{definition} 

\noindent On remarquera que tout opérateur faiblement $p$-nucléaire est
également faiblement${}^\ast$ $p$-nucléaire. De même tout opérateur
faiblement${}^\ast$ nucléaire d'un espace réflexif $E$ dans un espace $F$
est faiblement nucléaire.

\begin{proposition}\label{prop:faiblement-1-nuc-equals-HS}
  Soit $u\in\mathcal{L}\left(H_1,H_2\right)$. Alors $u$ est faiblement${}^\ast$ $1$-nucléaire si et seulement si $u$ est de Hilbert-Schmidt.
\end{proposition}
\begin{proof}
  Soit $u \in\mathcal{L}\left(H_1,H_2\right)$ un opérateur de
  Hilbert-Schmidt. On a $u=\sum_n \lambda_n e_n^*{\otimes} e_n$, avec
  $(e_n^*)$ et $(e_n)$ des familles orthonormales de $H_1$ et $H_2$,
  et $(\lambda_n) \in \ell_2$. On pose $x_n^*=\lambda_n e_n^*$. On a,
  pour $x \in H_1$,
\begin{align*}
\sum_{n=1}^{+\infty}\left|\left\langle x_n^*,x\right\rangle\right|\
=& \; \sum_{n=1}^{+\infty}\left|\left\langle \lambda_n e_n^*,x\right\rangle\right|
\; =  \; \sum_{n=1}^{+\infty}\left|\lambda_n\right|\left|\left\langle e_n^*,x\right\rangle\right|\\
\leq& \; \left(\sum_{n=1}^{+\infty}\left|\lambda_n\right|^2\right)^{\demi}\left(\sum_{n=1}^\infty\left|\left\langle e_n^*,x\right\rangle\right|^2\right)^{\demi}
\; \leq \; \left\|x\right\|\left(\sum_{n=1}^{+\infty}\left|\lambda_n\right|^2\right)^{\demi}
\end{align*}
Donc $u=\sum _n x_n^*{\otimes} e_n$ avec $(e_n)_{n\ge
  1}\in\ell_\infty(H_2)$ et $(x_n^*)_{n\ge
  1}\in\ell_1^{\FAIBLE\ast}(H_1)$. Par définition, $u $ est donc faiblement$\ast
1$-nucléaire.

\smallskip

\noindent Réciproquement, soit $u=\sum_n x_n^* {\otimes} y_n$ avec
$(x_n^*)_{n\ge 1} \in \ell_1^{\FAIBLE\ast}(H_1)$ et $(y_n)_{n\ge 1}
\in \ell_{\infty}(H_2)$ un opérateur faiblement $1$-nucléaire. On a $u
h =\sum _n \left\langle h ,x_n^*\right\rangle y_n$.  On considère les
opérateurs \ $ v : H_1\rightarrow \ell_1$ , $h\longmapsto
(\left\langle h, x_n^*\right\rangle)_{n\ge 1}$ et $w
:\ell_1\rightarrow H_2$ , $(\alpha_n)_{n\ge 1}$ $\longmapsto \sum _n
\alpha_n y_n$. Donc $u=w\circ v$. On voit bien que $u$ se factorise à
travers l'espace $\ell_1$ et d'après le
Corollaire~\ref{cor:HS_equivalences}, $u$ est de Hilbert-Schmidt.
\end{proof}

\subsection{Opérateurs $\gamma$-sommants et $\err$-sommants}


\noindent L'orthonormalisation d'une suite indépendante de Gaussiennes
ou de variables Rademacher implique que
\[
   \EE \Bigl| \sum_n r_n \langle x_n, x^* \rangle \Bigr|^2 = \EE \Bigl| \sum_n \gamma_n \langle x_n, x^* \rangle \Bigr|^2 = \sum_n \bigl|\langle x_n, x^* \rangle\bigr|^2
\]
ce qui a pour effet que $\GAUSSbdd^\FAIBLE(E) = \RADbdd^\FAIBLE(E) =
\ell_2^\FAIBLE(E)$, o\`u $\GAUSSbdd^\FAIBLE(E)$ et  $\RADbdd^\FAIBLE(E)$ sont les analogues "faibles" des espaces $\GAUSSbdd(E)$ et  $\RADbdd(E)$ introduits à la section 2. Ceci motive la définition suivante :

\begin{definition}[Opérateurs $\gamma$-sommants et $\err$-sommants]\label{def:gamma-summing and R-summing}
  Soit $u \in {\mathcal L}(E; F)$. Alors $u$ est appelé Gauss-sommant ou
  $\gamma$-sommant si l'image de toute suite $(x_n)_n \in\ell_2^\FAIBLE(E)$
  appartient à $\GAUSSbdd(F)$.  

\smallskip

  \noindent De même, $u$ est appelé
  Rademacher-sommant ou $\err$-sommant, si l'image de toute suite
  $(x_n)_n \in \ell_2^\FAIBLE(E)$ appartient à $\RADbdd(F)$. On note alors $u \in
  \gamma^\infty(E; F)$, respectivement $u \in \err^\infty(E; F)$.
\end{definition}

\noindent Les opérateurs $\gamma$-sommant ou $R$-sommants
sont évidemment continus. Par le théorème du graphe fermé on obtient dans les
deux cas l'existence d'une constante $C\geq 0$ vérifiant, pour toute
famille finie $\left(x_1,\dots,x_N\right)$ d'éléments de $E$,
\begin{equation}\label{eq gamma} 
  \Bigl(\EE\Bigl\|\sum_{n=1}^N\gamma_n u x_n\Bigr\|_F^2\Bigr)^\demi\leq C \sup_{x^*\in B_{E^*}}\Bigl(\sum_{n=1}^N\left|\left\langle x_n ,x^*\right\rangle\Bigr|^2\right)^\demi
\qquad \text{si } \, u \in \gamma^\infty(E; F),
\end{equation}
respectivement
\begin{equation}\label{eq R} 
  \Bigl(\EE\Bigl\|\sum_{n=1}^N r_n u x_n\Bigr\|_F^2\Bigr)^\demi\leq C \sup_{x^*\in B_{E^*}}\Bigl(\sum_{n=1}^N\left|\left\langle x_n ,x^*\right\rangle\right|^2\Bigr)^\demi
\qquad \text{si } \, u \in \err^\infty(E; F),
\end{equation}
On note alors $\norm{ u }_{\gamma^\infty}$ respectivement $\norm{ u
}_{R^\infty}$ les meilleures constantes dans (\ref{eq gamma}),
respectivement (\ref{eq R}). Remarquons que l'inégalité de Kahane
(Théorème~\ref{thm:kahane}) permet de passer à une norme équivalente
en remplaçant les normes $\ELL_2$ en normes $\ELL_p$ sur le coté gauche de
(\ref{eq gamma}) et (\ref{eq R}).

\bigskip

\begin{remark}
  Dans le cas où $E=H$ est un espace Hilbertien, l'espace
  $\gamma_2^\infty(H; F)$ est parfois défini dans la littérature par
  la propriété
\begin{equation}  \label{eq:gamma-summing-equiv}
    \sup_{\Lambda}\;   \EE \Bignorm{\sum_{n} \gamma_n u \, h_n }^2 \; \leq \; C^2.
\end{equation}
où le supremum est pris sur l'ensemble $\Lambda$ de tous les systèmes orthonormaux 
finis $(h_n)_n$ d'éléments de $H$\footnote{La définition vise de couvrir des
  espaces de Hilbert non-séparables ; dans le cas d'un Hilbert séparable de dimension infinie
  on peut se limiter aux familles $\{e_1,\dots,e_N\}$ pour $N\ge 1$, o\`u $(e_n)_{n\ge 1}$ désigne une base hilbertienne de $H$}.
Les deux notions coïncident: en
effet, notons que toute suite orthonormale $(e_n)_{n\ge 1}$ définit une suite
de $\ell_2^\FAIBLE(H)$; Ainsi (\ref{eq gamma}) implique
(\ref{eq:gamma-summing-equiv}) avec $C = \norm{ u
}_{\gamma^\infty}$. Réciproquement, supposons
(\ref{eq:gamma-summing-equiv}) vérifié pour tout système orthonormal fini.  On
sait alors (voir, par exemple \cite[Théorème 6.2]{jvn}), que $u \circ
v$ satisfait (\ref{eq:gamma-summing-equiv}) pour tout opérateur $v \in
\BOUNDED(H)$. Pour montrer (\ref{eq gamma}), soit $(x_n) \in
\ell_2^\FAIBLE(H)$ donné. Pour un système orthonormé $(h_n)$ de $H$ on
pose $v(h_n) := x_n$. Alors $v$ définit un opérateur linéaire continu
sur $H$ qui satisfait $\norm{v} = \norm{(x_n)_{n
    \in\N}}_{\ell_2^\FAIBLE(H)}$ : on en déduit (\ref{eq gamma}) par
la propriété d'idéal.
\end{remark}

Observons que les suites $\gamma$-sommables (respectivement
Rademacher-sommables) peuvent être confondues avec les opérateurs
$\gamma$-sommants (respectivement Rademacher-sommants) $u: \ell_2 \to
E$; de ce fait, par abus de notation, on peut écrire
$\gamma^\infty(\ell_2; E) = \gamma^\infty(E)$ et 
$\err^\infty(\ell_2; E) = \err^\infty(E)$ respectivement.

\medskip

A partir de (\ref{eq gamma}) et (\ref{eq R}) on vérifie aussitôt que
$\gamma^\infty(E: F)$ et $\err^\infty(E; F)$ vérifient la propriété
d'idéal: si $E$, $F$, $X$, $Y$ sont des espaces de Banach et si
$v\in\mathcal{L}\left(X;E\right)$, $w\in\mathcal{L}\left(F;Y\right)$
et $u \in \gamma^\infty(E; F)$ (respectivement $u\in
\err^\infty\left(E;F\right)$), alors la composition $w\circ u\circ v$
est $\gamma$-sommant respectivement $\err$-sommant et on a
\[
\left\|w\circ u\circ  v\right\|_{\gamma^\infty\left(X;Y \right)}
\leq
\left\|v\right\|\left\|u\right\|_{\gamma^\infty\left(E;F\right)}\left\|w\right\|,
\]
respectivement 
\[
\left\|w\circ u\circ  v\right\|_{\err^\infty\left(E;F\right)} 
\leq
\left\|v\right\|\left\|u\right\|_{\err^\infty\left(E;F\right)}\left\|w\right\|.
\]

\medskip

On déduit du "principe de comparaison", voir \cite[Théorème~3.2]{vdn}
que $\gamma^\infty(F)\subset \err^\infty(F)$. Si on pose
$m_1=\sqrt{\nicefrac{2}{\pi}}$ on a précisément 
\[
 \left \Vert (y_n) \right \Vert _{\err^\infty(F)} 
\le  m_1^{-1}
\left \Vert  (y_n) \right \Vert _{\gamma^\infty(F)}
\]
pour $(x_n)_{n\ge 1} \in \gamma^\infty(F)$. Ceci est une conséquence
facile du principe de contraction de Kahane et du fait que les
variables aléatoires $\gamma_n$ et $r_n |\gamma_n|$ ont la même
distribution.  En général, les espaces $\gamma^\infty(F)$ et
$\err^\infty(F)$ sont distincts, mais le Lemme~12.14 de \cite{DJT}
donne des "comparaisons en moyenne" qui permettent de démontrer le
résultat suivant.

\begin{theorem} \label{thm:gamma-somm-egal-err-somm}
  Soient $E$ et $F$ deux espaces de Banach, et soit $u \in {\mathcal
    L}(E; F)$. Alors $u$ est Rademacher-sommant si et seulement si $u$
  est $\gamma$-sommant, et dans ce cas on a
\[
    m_1\Vert u \Vert_{\err^\infty(E; F)}\le \Vert u \Vert_{\gamma^{\infty}(E; F)}\le \Vert u \Vert_{\err^\infty(E; F)}.
\]
\end{theorem}
\begin{proof} \def\calO{\mathcal O}
  Seule la deuxième estimation est à montrer. On suppose donc que $u
  \in {\mathcal L}(E; F)$ est $\err$-sommant et se donne une suite
  $(x_k) \in \ell_2^\FAIBLE(E)$. Définissons $v \in \BOUNDED(\ell_2;
  E)$ par $v(e_k) = x_k$: on a alors $\norm{v} = \norm{ (x_k)
  }_{\ell_2^\FAIBLE(E)}$. En notant $\calO_n$ le groupe orthogonal de
  $\ell_2^n$, et $\sigma_n$ la mesure de Haar normalisé sur $\calO_n$
  on a par le lemme 12.14 de \cite{DJT} 
  \begin{align*}
\EE \Bignorm{ \sum_{k=1}^n \gamma_k u k_k }^2 
= &\;  \EE \Bignorm{ \sum_{k=1}^n \gamma_k u\circ v \circ w e_k }^2\\
= &\;  \int_{\calO_n} \EE \Bignorm{ \sum_{k=1}^n \gamma_k u\circ v \circ w e_k }^2\,d\sigma_n(w)\\
\le&\; \norm{u}_{\err^\infty(E; F)} \int_{\calO_n} \sup_{\norm{x^*}\le 1} \sum_{k=1}^n \bigl \langle v \circ w e_k, x^* \rangle\bigr|^2\,d\sigma_n(w)\\
\le&\; \norm{u}_{\err^\infty(E; F)} \sup_{w\in \calO_n} \sup_{\norm{x^*}\le 1} \sum_{k=1}^n \bigl \langle e_k, w^* v^* x^* \rangle\bigr|^2\,d\sigma_n(w) \\
 = & \; \norm{u}_{\err^\infty(E; F)} \norm{ v }^2 = \norm{u}_{\err^\infty(E; F)} \norm{  (x_k)   }_{\ell_2^\FAIBLE(E)}^2.\qedhere
  \end{align*}
\end{proof}

\bigskip

\subsection{Opérateurs $\gamma$-radonifiants et opérateurs presque sommants}
On déduit du théorème d'Itô-Nisio, voir \cite[Théorème~2.17]{jvn} que l'on a les résultats suivants.

\begin{theorem} \label{thm:ito-nisio} Soit $(x_n)_{n\ge 1}$ une suite
  $\gamma$-sommable (respectivement Rademacher-sommable) d'éléments
  d'un espace de Banach $E$. Les conditions suivantes sont
  équivalentes.
   \begin{aufzi}
 \item  La série  $\sum_n\gamma_nx_n$ (respectivement $\sum_n r_nx_n$) converge presque sûrement  dans $E$.
 \item  La série $\sum_n \gamma_nx_n$ (respectivement $\sum_n r_n x_n$) converge en probabilité dans $E$.
 \item Il existe $p\ge 1$ tel que la série $\sum_n \gamma_nx_n$
   respectivement (respectivement $\sum_n r_n x_n$) converge dans
   $L^p(\Omega, E)$.
 \item La série $\sum_n \gamma_nx_n$ (respectivement $\sum_n r_n x_n$)
   converge dans $L^p(\Omega, E)$ pour tout $p\ge 1$.
 \end{aufzi}
 De plus l'ensemble $\gamma(E)$ (respectivement $\err(E)$) des suites
 $(x_n)_{n\ge 1}$ vérifiant ces conditions est égal à l'adhérence de
 l'ensemble $c_{00}(E)$ des suites d'éléments de $E$ nulles à partir
 d'un certain rang dans $\gamma^\infty(E)$ respectivement
 $\err^\infty(E)$, et on a, pour $(x_n) _{n\ge 1}$,
\[
    \left \Vert (x_n)_{n\ge 1}\right \Vert_{\gamma^\infty(E)} =\left (\EE \left \Vert \sum_{n=1}^{+\infty}\gamma_nx_n\right \Vert^2\right)^{\demi},
\]
respectivement
 \[
    \left \Vert (x_n)_{n\ge 1}\right \Vert_{\err^\infty(E)} =\left (\EE \left \Vert \sum_{n=1}^{+\infty}r_nx_n\right \Vert^2\right)^{\demi}.
 \]
\end{theorem}

\noindent Ceci suggère les notions suivantes.

\begin{definition} On dit qu'un opérateur $u: E\to F$ est
  $\gamma$-radonifiant (respectivement presque sommant)) si $(ux_n)_{n\ge 1}\in
  \gamma(F)$ (respectivement  $(ux_n)_{n\ge 1}\in \err(F)$) pour toute suite
  $(x_n)_{n\ge 1} \in \ell_2^\FAIBLE(E)$. L'ensemble des opérateurs
  $\gamma$-radonifiants (respectivement presque sommants) $u:E\to F$ est noté
  $\gamma(E; F)$ (respectivement $\Pi_{ps}(E; F))$.  Pour $u \in \gamma(E; F)$
  (respectivement $\Pi_{ps}(E; F)$) on pose
\[
    \Vert u \Vert_{\gamma(E; F)} =\Vert u \Vert_{\gamma^{\infty}(E; F)} \quad \mbox{respectivement} \quad  \pi_{ps} (u) =\Vert u \Vert_{\err^\infty(E; F)}.
\]  
\end{definition}

\noindent Notons que si $u \in \BOUNDED(E; F)$ est $\gamma$-radonifiant, alors
le Théorème~\ref{thm:ito-nisio} implique que $\Vert u
\Vert_{\gamma(E; F)}$ est la plus petite constante $c \ge 0$ telle que
\[
\EE \Bignorm{ \sum_n \gamma_n u \, x_n }^2 \le \; c^2 \, \bignorm{   (x_n) }^2_{\ell_2^\FAIBLE(E)}
\]
pour toute suite $(x_n)_{n\ge 1} \in {\ell_2^\FAIBLE(E)}$, et de même
pour des opérateurs presque sommants.


\noindent On vérifie que $\gamma(E; F)$ contient l'adhérence dans
$\gamma^{\infty}(E; F)$ de l'ensemble des opérateurs de rang fini, et
ces deux ensembles sont égaux si $H$ est un espace de Hilbert
séparable. De même $\Pi_{ps}(E; F)$ contient l'adhérence dans
$\err^\infty(E; F)$ de l'ensemble des opérateurs de rang fini.

\smallskip

Le résultat suivant est une reformulation d'un résultat de
Hoffmann-J\o{}rgensen et Kwapien, \cite{HK,Kwap}. Nous complétons la
version de ce théorème donné par van~Neerven dans
\cite[Théorème~4.2]{jvn} en incorporant à l'énoncé une conséquence
d'un exemple donné par Linde et Pietsch dans \cite{Pietsch} d'un
exemple d'opérateur $\gamma$-sommant $T\in \BOUNDED(\ell_2; c_0)$ qui
n'est pas $\gamma$-radonifiant, voir aussi \cite[Exemple 4.4]{jvn}.

\begin{theorem}
Soit $E$ un espace de Banach. Les conditions suivantes sont équivalentes.
\begin{aufzi}
  \item $\gamma(E)=\gamma^\infty(E)$.
  \item $\err(E)=\err^\infty(E)$.
  \item L'espace $E$ ne contient aucun sous-espace fermé isomorphe
    à $c_0$.
  \end{aufzi}
\end{theorem}

\noindent En combinant ce résultat avec le
théorème~\ref{thm:gamma-somm-egal-err-somm} on déduit immédiatement

\begin{corollary} Soit $F$ un espace de Banach. Si $F$ ne contient
  aucun sous-espace fermé isomorphe à $c_0$, alors
  $\gamma^{\infty}(E; F)=\gamma(E; F)=\err^\infty(E; F)=\Pi_{ps}(E; F)$, et on a,
  pour tout opérateur $\gamma$-radonifiant $u:E\to F$
\[
    m_1\pi_{ps}(u)\le \Vert u \Vert_{\gamma(E; F)}\le \pi_{ps}(u).
\]
\end{corollary}

\noindent Dans le cas général on a $\gamma(E; F)\subset \Pi_{ps}(E; F)$,
et $m_1\pi_{ps}(u)\le \Vert u \Vert_{\gamma(E; F)}$ pour $u \in
\gamma(E; F)$.  Notons que dans le chapitre 12 de \cite{DJT} les
auteurs ont malheureusement confondu la classe des opérateurs presque
sommants avec la classe des opérateurs Rademacher bornés, comme l'ont
remarqué avant nous Blasco, Tarieladze et Vidal dans \cite{BTV}. Nous
reviendrons sur cette question dans \cite{AEH}. 
On a le résultat
classique suivant

\begin{proposition} 
  Soient $H_1$ et $H_2$ deux espaces de Hilbert. Alors
  $S_2(H_1,H_1)=\gamma^{\infty}(H_1,H_2)=\gamma(H_1,H_2)$. De plus
  $\pi_2(u)=\Vert u\Vert_{\gamma(H_1,H_2)}$ pour tout opérateur de
  Hilbert-Schmidt $u$.
\end{proposition}

\begin{proof} On utilise l'orthonormalité des suites gaussiennes et le fait que $\norm{x}^2_H = \langle x, x \rangle$.
\end{proof}

\noindent Finalement, on rappelle un résultat de Linde et Pietsch
\cite{lp,Pietsch} sur la comparaison entre opérateurs $p$-sommants et
$\gamma$-sommants. La preuve repose sur le théorème de factorisation
de Pietsch et les inégalités de Khintchine-Kahane, voir par exemple
\cite[Proposition 12.1]{jvn}.

\begin{proposition} \label{prop:p-sommant-implique-radonifiant}
Soit $p\ge 1$. Alors tout opérateur $p$-sommant
 est $\gamma$-radonifiant, et on a
\[
    \Vert u \Vert_{\gamma(E; F)}\le  c_p\, \pi_p(u) \ \ (u \in \Pi_{p}(E; F)),
\]
avec $c_p=K^{\gamma}_{2,p}K^{\gamma}_{p,2}$ où $K^{\gamma}_{2,p}$ et
$K^{\gamma}_{p,2}$ sont les constantes intervenant dans les
inégalités de Kahane-Khintchine gaussiennes.
\end{proposition}

\bigskip

\section{Type et  cotype des espaces de Banach}
\noindent On rappelle les notions suivantes, qui sont apparues à la suite des
travaux de Hoffmann-J\o{}rgensen, Kwapień, Maurey et Pisier au
début des années 1970, voir par exemple \cite{mau} et les
exposés du Séminaire Maurey-Schwartz de cette époque à l'École
Polytechnique.

\begin{definition}
  \begin{aufzi}
  \item On dit qu'un espace de Banach $E$ est de type $p\in [1,2]$
    s'il existe une constante $c_p>0$ telle que pour toute famille
    finie $x_1,\dots,x_N$ d'éléments de $E$ on ait
\[
    \left (\EE  \left \Vert \sum \limits_{n=1}^Nr_nx_n\Vert ^2\right \Vert \right)^{\demi} \le c_p\left (\sum \limits_{n=1}^N\Vert x_n\Vert^p\right )^{\nicefrac{1}{p}}.
\]
   \item On dit qu'un espace de Banach $E$ est de cotype $q\in
  [2,+\infty)$ s'il existe une constante $c'_q>0$ telle que pour toute
  famille finie $x_1,\dots,x_N$ d'éléments de $E$ on ait
\[
    \left (\sum \limits_{n=1}^N\Vert x_n\Vert^q \right )^{\nicefrac{1}{q}}\le c'_q \left (\EE \left \Vert \sum \limits_{n=1}^N r_nx_n\right \Vert ^2 \right )^{\demi}.
\]
\end{aufzi}
\end{definition}

\noindent Les plus petites constantes dans les définitions ci-dessus
sont notées $T^r_p(E)$ et $C^r_q(E)$. On vérifie que tout espace de
Banach est de type 1 (on appellera ceci le type trivial) et, moyennant
une adaptation évidente de la définition, de cotype infini. On peut
définir de même le type et le cotype en utilisant une suite gaussienne
au lieu d'une suite de Rademacher. On obtient les mêmes notions, voir
\cite[Proposition 12.11 et Lemma 12.1]{DJT}, avec des constantes
gaussiennes $T^{\gamma}_p(E)$ et $C^{\gamma}_q(E)$. On vérifie qu'un
espace de Banach a les mêmes types et cotypes que son bidual, voir
\cite[Proposition 11.9]{DJT}.

\noindent On rappelle que pour $1<p<\infty$ on note $p^*$ le conjugué de $p$,
défini par la formule $1/p^*+1/p=1$. On a le résultat facile
suivant, voir\cite[Proposition 11.10]{DJT}.

\begin{proposition} 
  Si un espace de Banach $E$ est de type $p\in (1,2]$, alors son dual
  $E^*$ est de cotype $p^*, $ et $C_{p^*}^r(E^*)\le T_p^r(E)$.
\end{proposition}

\noindent La réciproque est évidemment fausse, puisque $\ell_1$ est de
cotype $2$ alors que son prédual $c_0$ n'est de type $p$ pour aucun
$p>1$. De même le fait que $E$ soit de cotype fini n'implique rien
sur le type de $E^*$, puisque $\ell_1$ est de cotype 2 alors que son
dual $\ell_{\infty}=c_0^{**}$ n'est de type $p$ pour aucun $p>1$. Par
contre puisque $E$ et $E^{**}$ ont même type et cotype le fait que
$E^*$ soit de type $p$ implique que $E$ est de cotype $p^*$.

\noindent On vérifie que pour $1\le p <+\infty$ un espace $\ELL_p$ de
dimension infinie est de type $\min(p,2)$ et de cotype $\max(p,2), $ et
que ces résultats sont optimaux. Un espace $\ELL_{\infty}$ de
dimension infinie ne peut être de type non-trivial ou de cotype
fini. Il était d'ailleurs à priori évident que si $E$ contient une
copie de $c_0$ alors $E$ n'est pas de cotype fini puisque si
$(e_n)=(\delta_{m,n})_{m\ge 1}$ alors
\[
\left (\sum   \limits_{n=1}^N\Vert r_ne_n\Vert^q \right )^{\nicefrac{1}{q}}=N
\quad\text{tandis que}\quad
\left (\EE \sum \limits_{n=1}^N\Vert r_ne_n\Vert ^2 \right )^{\demi}=1.
\]
\noindent Un espace de Hilbert est à la fois de type 2 et de cotype 2, et un
résultat profond de Kwapień \cite{K2} montre que réciproquement
les seuls espaces de Banach qui sont à la fois de type 2 et de
cotype 2 sont les espaces isomorphes aux espaces de Hilbert.

\smallskip

On a le résultat suivant, voir \cite[Proposition 2.6]{jvn}, qui montre que
les espaces de suites $\gamma$-sommables et Rademacher-sommables
d'éléments de $E$ coïncident si $E$ est de cotype
fini.

\begin{proposition} 
  Si $E$ est de cotype fini, alors pour $1 \le p <+\infty$ il existe
  une constante positive $C_{p,E}$ telle que pour toute famille finie
  $(x_1,\dots, x_N)$ d'éléments de $E$ on ait
\[
     \EE \left \Vert \sum \limits_{n=1}^{N} \gamma_nx_n\right \Vert^p\le C_{p,E} \EE \left \Vert \sum \limits_{n=1}^{N} r_nx_n\right \Vert^p.
\]
\end{proposition}

\noindent Par conséquent on a trivialement $\gamma(E; F)=\Pi_{ps}(E; F)$ si $F$
est de cotype fini, ce qui résulte aussi du fait qu'un espace de
cotype fini ne contient aucun sous-espace fermé isomorphe à $c_0$.
Le résultat suivant est dû à Linde et Pietsch, voir 
\cite{lp,Pietsch,mau}.

\begin{theorem} \label{thm:LP-cotype2}
  Soit $F$ un espace de Banach. Alors les conditions suivantes sont
  équivalentes.
  \begin{aufzi}
  \item \label{item:LP-cotype2-c} $F$ est de cotype 2.
  \item \label{item:LP-cotype2-a} $\Pi_{2}(E; F)=\gamma^{\infty}(E; F)$ pour tout espace de Banach
    $E$.
  \item \label{item:LP-cotype2-b} $\Pi_{2}(H;F)=\gamma^{\infty}(H;F)$ pour tout espace de
    Hilbert $H$.
  \end{aufzi}
\end{theorem}

\begin{proof} Supposons que $F$ est de cotype 2, posons $c=C_{2,F}^r$,
  et soit $u \in \gamma^\infty(E; F)=\err^\infty(E; F)$. On a, pour toute
  famille finie $(x_1,\dots,x_N)$ d'éléments de $E$,
\[
    \sum_{n=1}^N\Vert ux_n\Vert^2\le c^2\EE\left \Vert \sum \limits _{n=1}^Nr_nux_n\right \Vert^2\le c^2 \Vert u \Vert_{\gamma(E; F)}\sup_{\Vert x^*\Vert\le 1}\sum \limits_{n=1}^N\left | \langle x_n,x \rangle \right  |^2.
\]
Donc $u \in \Pi_2(E; F)$, et $\pi_2(u) \le C^r_{2,F}\Vert u \Vert _{\err^\infty(E; F)}$.

Il est clair que \ref{item:LP-cotype2-a} implique
\ref{item:LP-cotype2-b}.  Supposons finalement que
\ref{item:LP-cotype2-b} est vérifié. Donc
$\err^\infty(\ell_2,F)=\gamma^{\infty}(\ell_2,F)=\Pi_2(\ell_2,F)$. Comme
les injections de $\gamma(\ell_2,F)$ et $\Pi_2(E; F)$ dans ${\mathcal
  L}(\ell_2,F)$ sont continues, le graphe de l'injection $i:
\gamma(\ell_2,F)\to \Pi_2(\ell_2,F)$ est fermé dans
$\gamma(\ell_2,F)\times \Pi_2(E; F)$, $i$ est continue, et il existe
une constante $c\ge 0$ telle que $\pi_2(u) \le c\Vert u \Vert
_{\gamma(\ell_2,F)}$ pour tout $u \in \gamma(\ell_2,F)$.
Soit $(x_1,\dots, x_N)$ une famille finie d'éléments de $E$,
$(\lambda_m)_{m\ge 1} \in \ell_2$.  Posons $u((\lambda_m))=\sum
\limits_{m=1}^N \lambda_mx_m$. Soit $(e_n)_{n\ge 1}$ la base
hilbertienne naturelle de $\ell_2$. On a
\begin{align*}
    \sum \limits_{n=1}^N\Vert x_n\Vert^2 = & \; \sum_{n=1}^N\Vert ue_n\Vert^2\le \pi^2_2(u)\le c^2\Vert u \Vert^2_{\gamma(l^2,F)}=\sup_{p\ge 1}\EE \left \Vert \sum_{n=1}^p\gamma_nue_n\right \Vert^2\\
    & \; =c^2 \EE \left \Vert \sum_{n=1}^N\gamma_nue_n\right \Vert^2=c^2\EE \left \Vert \sum_{n=1}^N\gamma_nx_n\right \Vert^2\le \tfrac{c^2\pi}{2}\EE \left \Vert \sum_{n=1}^Nr_nx_n\right \Vert^2.
\end{align*}
Donc $F$ est de cotype 2.
\end{proof}

\noindent Le résultat suivant, voir \cite [Corollaire 11.16]{DJT}, que nous
donnons sans démonstration, complète le théorème
précédent.

\begin{theorem}
  \begin{aufzi}
  \item Si $E$ est de cotype $2$, alors $\Pi_2(E; F)=\Pi_1(E; F)$.
  \item Si $E$ est de cotype $<q<+\infty$, alors
    $\Pi_r(E; F)=\Pi_1(E; F)$ pour $1\le r < q^*$.
\end{aufzi}
\end{theorem}

\smallskip

Pour $1 \le p \le \infty$, notons $\Gamma_p(E; F)$ l'ensemble des
opérateurs $u: E\to F$ admettant une factorisation de la forme
$u=v\circ w$, avec $v\in \BOUNDED(E;\ELL_p(\Omega, \mu)$ et $w \in
\BOUNDED(\ELL_p(\Omega, \mu);F)$ pour un certain espace mesuré
$(\Omega, {\mathcal B}, \mu)$. La proposition 2.12 de \cite{jvn}
montre que si $u: H \to F$ est $\gamma$-radonifiant, alors $u^*: H \to
F^*$ est $2$-sommant. Nous verrons plus loin un résultat un peu
plus général, valable pour les applications
$\gamma$-sommantes. Ces résultats admettent une réciproque si $E$
est de type 2. Ceci est une conséquence du théorème suivant,
voir \cite[Théorème~12.10]{DJT}.

\begin{theorem}\label{thm:equiv-type2} Soit $F$ un espace de Banach. Alors les conditions
  suivantes sont équivalentes
  \begin{aufzi}
  \item \label{item:equiv-type2-a} $F$ est de type 2.
  \item \label{item:equiv-type2-b} $\BOUNDED(\ell_1;F)=\gamma(\ell_1;F)$.
  \item \label{item:equiv-type2-c} $\Gamma_1(E; F)\subset \gamma(E; F)$ pour tout espace de Banach
    $E$.
\end{aufzi}
\end{theorem}

\noindent Ce résultat est énoncé dans \cite{DJT} en utilisant la
classe des opérateurs presque sommants, qui est confondue dans
\cite{DJT} avec la classe $\err^\infty(E; F)=\gamma^{\infty}(E; F)$ des
opérateurs Rademacher-sommants. Cette confusion n'est pas gênante dans
ce cas, puisqu'un espace de Banach de type 2 ne contient pas de
sous-espace fermé isomorphe à $c_0$.  Il est élémentaire que
\ref{item:equiv-type2-c} implique \ref{item:equiv-type2-b} qui
implique \ref{item:equiv-type2-a}, compte tenu du fait que
$\gamma_{\infty}(E; F)=\err^\infty(E; F)$ pour tout espace de Banach
$E$. Nous renvoyons à \cite[Chapitre 12]{DJT} pour la démonstration
que \ref{item:equiv-type2-a} implique \ref{item:equiv-type2-c}. Notons
que l'inégalité de Grothendieck joue un rôle essentiel dans cette
preuve: si $u: \ELL_1(\Omega, \mu) \to F$ est un opérateur borné, et si
$F$ est de type 2, il résulte de l'inégalité donnée à \cite[page
245]{DJT} que $u \in \gamma^{\infty}(\ELL_1(\Omega, \mu), F)$ et que l'on
a
\begin{equation} 
   \Vert u \Vert_{\gamma^{\infty}}( \ELL_1(\Omega, \mu); F)\le T^{\gamma}_2(F)K_G\Vert u \Vert 
\end{equation}
où $K_G$ désigne la constante de Grothendieck. On obtient alors le
résultat suivant, voir \cite[Corollaire 12.21]{DJT}.

\begin{corollary} \label{cor:u-ast-2-sommant-u-gamma} Soit $E$ un
  espace de Banach, soit $F$ un espace de Banach de type $2$, et soit
  $u\in \BOUNDED(E; F)$. Si $u^* \in \BOUNDED(F^{*}; E^*)$ est
  $2$-sommant, alors $u$ est $\gamma$-radonifiant.
\end{corollary}

\begin{proof} Il résulte du théorème de factorisation de Pietsch
  qu'il existe un compact $K$ et une mesure de probabilité sur $K$
  tels que $u^*=w\circ i \circ j$, où $j: {\mathcal C}(K)\subset
  \ELL_{\infty}(\mu) \to \ELL_2(\mu)$ est l'injection canonique, où
  $i:F^*\to {\mathcal C}(K)$ est isométrique et où $w: \ELL_2(K,\mu)
  \to E^*$ est continue. Mais on a $j=j_0^*$, où $j_0: \ELL_2(\mu)\to
  \ELL_1(\mu)$ est l'injection canonique. Donc $j_0^{**}(\ELL_2(\mu))\subset
  \ELL_1(\mu)$, et $u^{**} \in \Gamma_1(E^{**}; F^{**})\subset
  \gamma^{\infty}(E^{**},F^{**})$ ce qui implique que $u \in
  \gamma^{\infty}(E; F)=\gamma(E; F)$.
\end{proof}

\noindent Notons que l'hypothèse que $F$ est de type 2 est nécessaire
pour obtenir ce résultat: si $F$ n'est pas de type 2 il existe $u:
\ell_2 \to E$ non $\gamma$-sommant tel que $u^*: E^*\to \ell_2$ soit
$2$-sommant, voir \cite[Théorème~12.3]{jvn}.

\section{Extensions de Hilbert-Schmidt par relèvement}
\noindent On dira qu'un opérateur $u \in \BOUNDED(E; F)$ admet une
factorisation à travers un espace de Banach $Z$ s'il existe $v \in
\BOUNDED(E;Z)$ et $v \in \BOUNDED(Z;F)$ tels que $u = w\circ
v$, et on dira que $u$ est universellement factorisable s'il est
factorisable à travers $Z$ pour tout espace de Banach $Z$. Une
caractérisation profonde des opérateurs de Hilbert-Schmidt est
donné par le résultat suivant, voir \cite[Théorème~19.2]{DJT}.

\begin{theorem} Soit $u \in \BOUNDED(H_1;H_2)$. Alors $u$ est un
  opérateur de Hilbert-Schmidt si et seulement si $u$ est
  universellement factorisable.
\end{theorem} 

\noindent On introduit maintenant la notion d'espace de Hilbert-Schmidt, voir
\cite[Chapitre 19]{DJT}.

\begin{definition} Soit $E$ un espace de Banach. On dit que $E$ est un
  espace de Hilbert-Schmidt si tout opérateur $u:H_1\to H_2$
  admettant une factorisation à travers $E$ est un opérateur de
  Hilbert-Schmidt, et on note $\ESPACEHS$ la classe des espaces de
  Hilbert-Schmidt.
\end{definition}

\noindent Il est clair que $E \in \ESPACEHS$ si et seulement si $E^*\in
\ESPACEHS$. La classe des espaces de Hilbert-Schmidt contient les classes
$\ELL_1$ et $\ELL_{\infty}$, mais elle est beaucoup plus vaste. Par
exemple si $E$ est un sous-espace fermé de ${\mathcal C}(K)$ tel que
${\mathcal C}(K)/E$ soit réflexif, alors il résulte de
\cite[Théorème~15.13]{DJT} que ${\mathcal L}(E; F)=\Pi_2(E; F)$ pour
tout espace de Banach $F$ de cotype 2, et en particulier pour tout
espace de Hilbert, donc $E$ est un espace de Hilbert-Schmidt. De même
si $Z$ est un sous-espace réflexif de $\ELL_1(\mu)$ alors l'espace
quotient $\ELL_1(\mu)/Z$ a la "propriété de Grothendieck", c'est à dire
que tout opérateur $u: \ELL_1(\mu)/Z\to \ell_2$ est absolument sommant,
donc à fortiori $2$-sommant, et $\ELL_1(\mu)/Z$ est un espace de
Hilbert-Schmidt. Ainsi l'algèbre de Banach $H^{\infty}(\DD)$ des
fonctions holomorphes bornées sur le disque unité ouvert $\DD$, ainsi
que l'algèbre du disque ${\mathcal A}(\DD)$ formée des fonctions
holomorphes sur $\DD$ admettant un prolongement continu au disque
unité fermé sont en tant qu'espaces de Banach des espaces de
Hilbert-Schmidt.

\noindent  
Par contre aucun espace $K$-convexe ne peut être un espace
de Hilbert-Schmidt, voir \cite[page 443]{DJT}, ce qui signifie qu'aucun
espace de type non trivial ne peut être un espace de
Hilbert-Schmidt, d'après un célèbre théorème de Pisier
\cite{Pis2} ou \cite[Théorème~13.3]{DJT}.

\begin{definition}\label{def:PS2} On dit qu'un opérateur $u \in \BOUNDED(E;F)$ est pré-Hilbert-Schmidt si $w\circ u \circ v$ est un opérateur de
  Hilbert-Schmidt pour tout couple d'opérateurs bornés $v \in
  \BOUNDED(H_1; E)$ et $w \in \BOUNDED(F;H_2)$, où $H_1$ et $H_2$
  désignent des espaces de Hilbert quelconques.  L'ensemble des
  opérateurs pré Hilbert-Schmidt de $E$ dans $F$ sera noté $\PHS(E;
  F)$.
\end{definition}

\noindent Il est clair que $u \in \PHS(E; F)$ si et seulement si $u^*
\in \PHS(F^*;E^*)$, que $\PHS(E; F)=\BOUNDED(E; F)$ si $E$ ou $F$ est un
espace espace de Hilbert-Schmidt, et que $\PHS(H_1;H_2)= \HS(H_1;H_2)$
si $H_1$ et $H_2$ sont des espaces de Hilbert. L'observation suivante
est une reformulation d'un résultat bien connu, voir \cite[page
50]{jvn}.

\begin{proposition} \label{prop:phs-pi2}
Soit $E$ un espace de Banach, et $H$ un espace de Hilbert. Alors $\PHS(E;H)=\Pi_2(E;H)$.
\end{proposition}

\begin{proof} Par le
  corollaire~\ref{coro:sur-Hilbert-p-sommant-egal-HS} on a $\Pi_2(E;H)
  \subset \PHS(E;H)$. Soit $u\in \PHS(E;H)$, et soit $(x_n)_{n\ge
    1}\in \ell_2^{\FAIBLE}(E)$. On a, pour $N\ge M\ge 1$, et une suite
  $(\lambda_n)_{n\ge 1} \in \ell_2$,
\[
    \Bigl \Vert \sum \limits_{n=M}^N \lambda_nx_n \Bigr \Vert 
=   \sup_{x^* \in B_{E^*}} \Bigl | \sum \limits_{n=M}^N \lambda_n\langle x_n,x^* \rangle  \Bigr | 
\le \Bigl ( \sum \limits_{n=M}^N\lambda_n^2\Bigr )^{\demi}
    \left \Vert (x_n)_{n\ge 1}\right \Vert_{\ell_2^{faible}(E)}.
\]
la série $\sum_n \lambda_nx_n$ converge donc dans
$E$, et on obtient un opérateur borné $w: \ell_2\to E$ en posant
$u(\lambda_n)_{n\ge 1}=\UU \lambda_nx_n$ pour $(\lambda_n)_{n\ge 1}\in
\ell_2$. Soit $i \in \BOUNDED(H)$ l' application identité, et soit $(e_n)_{n\ge 1}$
la base hilbertienne naturelle de $\ell_2$. Alors $w\circ u=w \circ u
\circ i \in \HS(E;H)$ et on a
\[
\sum \limits _{n=1}^{\infty}\Vert ux_n\Vert^2 =\sum
\limits_{n=1}^{\infty} \Vert (w \circ u)e_n\Vert^2 =\Vert w\circ u
\Vert_{\HS(\ell_2;H)}<+\infty,
\]
ce qui montre que $u$ est $2$-sommant.
\end{proof}

\noindent Comme $u \in \PHS(E; F)$ si et seulement si $u^* \in \PHS(F^*;E^*)$, on a le corollaire suivant.
\begin{corollary} Soit $F$ un espace de Banach, et soit $H$ un espace
  de Hilbert. Alors $u\in \PHS(H;F)$ si et seulement si $u^*:F^*\to H$
  est $2$-sommant. En particulier si $u \in \gamma^{\infty}(H;F)$,
  alors $u^*$ est \ $2$-sommant.
\end{corollary}

\noindent En utilisant le théorème 4.20 de \cite{DJT} on obtient une
caractérisation des espaces de Hilbert-Schmidt mentionnée dans les
Notes du Chapitre 19 de \cite{DJT}.

\begin{theorem}\label{thm:char-PHS}  Soit $E$ un espace de Banach. Alors les conditions
  suivantes sont équivalentes.
  \begin{aufzi}
\item \label{item:char-PHS-a} $E$ est un espace de Hilbert-Schmidt.
\item \label{item:char-PHS-b} $u^*$ est $2$-sommant pour tout espace de Banach $F$ et pour    tout opérateur $2$-sommant $u \in \BOUNDED(F; E)$.
\item \label{item:char-PHS-c} $\BOUNDED(E;H)=\Pi_2(E;H)$ pour tout espace de Hilbert $H$.
\item \label{item:char-PHS-d} $u$ est $2$-sommant pour tout espace de Banach $F$ et pour tout  $u\in \BOUNDED(E; F)$ tel que $u^*$ est 2-sommant.
\item \label{item:char-PHS-e} $\BOUNDED(E^*;H)=\Pi_2(E^*;H)$ pour tout espace de Hilbert $H$.
\end{aufzi}
\end{theorem}

\begin{proof} L'équivalence de
  \ref{item:char-PHS-b}--\ref{item:char-PHS-e} est donnée par le
  théorème 4.20 de \cite{DJT}, et il est clair que si
  \ref{item:char-PHS-c} est vérifié alors tout opérateur $u:H_1\to
  H_2$ qui factorise à travers $E$ est $2$-sommant, donc de
  Hilbert-Schmidt. Ainsi, $E$ est un espace de
  Hilbert-Schmidt. Réciproquement si $E$ est un espace de
  Hilbert-Schmidt alors pour tout espace de Hilbert $H$,
  $\BOUNDED(E;H) = {\PHS}(E;H)$ et par la
  proposition~\ref{prop:phs-pi2}, $\PHS(E;H) = \Pi_2(E;H)$, donc
  \ref{item:char-PHS-c} est vérifié.
\end{proof}

\begin{theorem} 
  Soit $E$ un espace de Banach, et soit $F$ un espace de Banach de
  type 2. Alors $\PHS(E; F)=\gamma(E; F)$.
\end{theorem}
\begin{proof}
  Par le corollaire~\ref{coro:sur-Hilbert-p-sommant-egal-HS} et en
  remarquant que $\gamma(H) = \ell_2(H)$ on a $\gamma(E; F)\subset
  \PHS(E; F)$. Soit maintenant $u \in \PHS(E; F)$, et soit $v \in
  \BOUNDED(\ell_2;E)$. On a $u^*\in \PHS(F^*;E^*)$, donc $(u\circ
  v)^*=v^*\circ u^* \in \Pi_2(F^*;\ell_2)$ par la
  proposition~\ref{prop:phs-pi2}. Il résulte alors du
  corollaire~\ref{cor:u-ast-2-sommant-u-gamma} que $u\circ v: l^2 \to
  F$ est $\gamma$-radonifiant.

  \noindent Soit $(x_n)_{n\ge 1}\in \ell_2^{\FAIBLE}(E)$. De même que plus
  haut, on voit qu'il existe $v \in \BOUNDED(\ell_2;E)$ tel que
  $we_n=x_n$ pour tout $n\ge 1$, $\,(e_n)_{n\ge 1}$ désignant la base
  hilbertienne naturelle de $\ell_2$. Soit $(\gamma_n)_{n\ge 1}\subset
  \ELL_2(\Omega, {\mathbb P})$ une suite gaussienne. Alors la série $\sum
  _n \gamma_nx_n=\sum _n
  \gamma_n(u\circ v)e_n$ est convergente dans $\ELL_2(\Omega,{\mathbb
    P})$, ce qui montre que $u \in \gamma(E; F)$.
\end{proof}

\noindent Si $U \subset \BOUNDED(E; F)$, on pose $U^* \defequal \{ u^*\}_{u \in U}$.
Comme $\PHS(E; F)^*=\PHS(F^*;E^*)$, on obtient le corollaire suivant.

\begin{corollary} 
  Soit $E$ un espace de Banach tel que $E^*$ soit de type 2. Alors
  pour tout espace de Banach $F$ on a $\PHS(E; F)^*=\gamma(F^*;E^*)$.
\end{corollary}

\noindent Comme les espaces $\ell_p$ sont de type $\min(2, p)$, on obtient le
corollaire suivant.

\begin{corollary} Soient $E, F$ des espaces de Banach quelconques. Alors pour
$p \ge 2$ on a $\PHS(E;\ell_p)=\gamma(E;\ell_p)$ et pour $p\in [1,2];$ on a
    $\PHS(\ell_p;F)^*=\gamma(F^*;\ell_{p*})$
\end{corollary}

\noindent Notons que le corollaire permet de décrire concrètement
les éléments de diagonaux $\PHS(\ell_p;\ell_q)$ si $1 \le p \le 2$, ou
si $2\le q <+\infty$. En effet des calculs effectués par Linde et
Pietsch dans \cite{lp,Pietsch}, complétés par des calculs de Maurey
mentionnés dans \cite{lp}, donnent une description des opérateurs
diagonaux $\gamma$-sommants $u_{\sigma}: (x_n)_{n\ge 1}\to
(\sigma_nx_n)_{n\ge 1}\in \BOUNDED(\ell_p;\ell_q)$ associés à une
suite $\sigma=(\sigma_n)_{n\ge 1}$ : on obtient le tableau suivant,
qui caractérise les suites $\sigma$ telles que $u_{\sigma} \in
\gamma(\ell_p;\ell_q)$. et dans certains cas les suites $\sigma$
telles que $u_{\sigma} \in \PHS(\ell_p; \ell_q)$

\begin{center} 
  \begin{tabular}{| l| c| r |}\hline $p$ & $q$
    &$u_{\sigma}\in \gamma(\ell_p;\ell_q)$ \cr \hline $1\le p < 2 $&$
    1\le q < \nicefrac{2p}{2-p}$ & $\sigma \in \ell_r, \nicefrac{1}{r}= \nicefrac{1}{2}-\nicefrac{1}{p} +\nicefrac{1}{ q}$ \cr \hline $1\le p < 2$ &$ q \ge
    \nicefrac{2p}{2-p}$ & $\sigma \in \ell_{\infty}$ \cr \hline $2\le p <
    +\infty$&$ q \ge 1$ & $\sigma \in \ell _q$ \cr
    \hline \end{tabular} 
\end{center}

\medskip

\noindent En appliquant ces résultats à $u_{\sigma}$ et à $u_{\sigma}^*\in
\BOUNDED(\ell_{q^*};\ell _{p^*})$, on peut caractériser les suites
$\sigma$ telles que $u_{\sigma} \in \PHS(\ell_p; \ell_q)$ pour $1\le p
< 2$, $1 \le q <+\infty$ et pour $2 \le p <+\infty, 2 \le q
<+\infty$ (nous renvoyons à \cite{lp,Pietsch} pour le cas o\`u
$p=+\infty$ et/ou $q={+\infty}$).

\medskip

\begin{center} 
  \begin{tabular}{| l| c| r |}\hline $p$ & $q$
    &$u_{\sigma}\in \PHS(\ell_p;\ell_q)$ \cr \hline $1\le p < 2 $&$
    1\le q < \nicefrac{2p}{2-p}$ & $\sigma \in \ell_r, \nicefrac{1}{ r}=\nicefrac{1}{2}-\nicefrac{1}{ p} +\nicefrac{1}{ q}$ \cr \hline $1\le p < 2$ &$ q \ge
    \nicefrac{2p}{ 2-p}$ & $\sigma \in \ell_{\infty}$ \cr \hline $2\le p <
    +\infty$&$ q \ge 2$ & $\sigma \in \ell _q$ \cr
    \hline 
  \end{tabular} 
\end{center}

\medskip

Nous concluons cet cet article par la conjecture suivante, qui est
vérifiée d'après le théorème de factorisation de Pietsch par tout
opérateur $u$ tel que $u$ ou $u^*$ soit $p$-sommant avec $p \le 2.$

\begin{conjecture} 
  Soient $E$ et $F$ deux espaces de Banach. Alors tout opérateur
  pré-Hilbert-Schmidt $u:E\to F$ se factorise à travers un espace
  de Hilbert-Schmidt.
\end{conjecture}

\end{document}